\documentclass[11pt]{amsart}
\usepackage{amssymb,amsmath,amsthm,amsfonts} 
\usepackage{enumerate}
\usepackage {latexsym}
\usepackage{graphicx}
\usepackage{epsfig}
\usepackage{bbm}
\include{srctex}
\usepackage{marginnote}
\usepackage[letterpaper, top=1.2in, bottom=1.2in, outer=1.1in, inner=1.1in, heightrounded, marginparwidth=2.5cm, marginparsep=2mm]{geometry}
  
\allowdisplaybreaks
\newcommand{\nc}{\newcommand}
\nc{\les}{\lesssim}
\nc{\nit}{\noindent}
\nc{\nn}{\nonumber}
\nc{\D}{\partial}
\nc{\diff}[2]{\frac{d #1}{d #2}}
\nc{\diffn}[3]{\frac{d^{#3} #1}{d {#2}^{#3}}}
\nc{\pdiff}[2]{\frac{\partial #1}{\partial #2}}
\nc{\pdiffn}[3]{\frac{\partial^{#3} #1}{\partial{#2}^{#3}}}
\nc{\abs}[1] {\lvert #1 \rvert}
\nc{\cAc}{{\cal A}_c}
\nc{\cE}{{\cal E}}
\nc{\cF}{{\cal F}}
\nc{\cP}{{\cal P}}
\nc{\cV}{{\cal V}}
\nc{\cQ}{{\cal Q}}
\nc{\cGin}{{\cal G}_{\rm in}}
\nc{\cGout}{{\cal G}_{\rm out}}
\nc{\cO}{{\cal O}}
\nc{\Lav}{{\cal L}_{\rm av}}
\nc{\cL}{{\cal L}}
\nc{\cB}{{\cal B}}
\nc{\cZ}{{\cal Z}}
\nc{\mR}{{\mathcal R}}
\nc{\mG}{{\mathcal G}}
\nc{\cT}{{\cal T}}
\nc{\cY}{{\cal Y}}
\nc{\cX}{{\cal X}}
\nc{\cXT}{{{\cal X}(T)}}
\nc{\cBT}{{{\cal B}(T)}}
\nc{\vD}{{\vec \mathcal{D}}}
\nc{\efield}{\mathcal{E}}
\nc{\vE}{{\vec \efield}}
\nc{\vB}{{\vec \mathcal{B}}}
\nc{\vH}{{\vec \mathcal{H}}}
\nc{\F}{  \mathcal{F} }
\nc{\ty}{{\tilde y}}
\nc{\tu}{{\tilde u}}
\nc{\tV}{{\tilde V}}
\nc{\Pc}{{\bf P_c}}
\nc{\bx}{{\bf x}}
\nc{\bX}{{\bf X}}
\nc{\bXYZ}{{\bf XYZ}}
\nc{\bY}{{\bf Y}}
\nc{\bF}{{\bf F}}
\nc{\bS}{{\bf S}}
\nc{\dV}{{\delta V}}
\nc{\dE}{{\delta E}}
\nc{\TT}{{\Theta}}
\nc{\dPsi}{{\delta\Psi}}
\nc{\order}{{\cal O}}
\nc{\Rout}{R_{\rm out}}
\nc{\eplus}{e_+}
\nc{\eminus}{e_-}
\nc{\epm}{e_\pm}
\nc{\eps}{\varepsilon}
\nc{\vnabla}{{\vec\nabla}}
\nc{\G}{\Gamma}
\nc{\w}{\omega}
\nc{\mh}{h}
\nc{\mg}{g}
\nc{\vphi}{\varphi}
\nc{\tlambda}{\tilde\lambda}
\nc{\be}{\begin{equation}}
\nc{\ee}{\end{equation}}
\nc{\ba}{\begin{eqnarray}}
\nc{\ea}{\end{eqnarray}}

\nc{\g}{\gamma}
\nc{\ol}{\overline}

\newtheorem{theorem}{Theorem}[section]
\newtheorem{lemma}[theorem]{Lemma}
\newtheorem{prop}[theorem]{Proposition}
\newtheorem{corollary}[theorem]{Corollary}
\newtheorem{defin}[theorem]{Definition}
\newtheorem{rmk}[theorem]{Remark}

\nc{\pT}{\partial_T}
\nc{\pz}{\partial_z}
\nc{\pt}{\partial_t}
\nc{\la}{\langle}
\nc{\ra}{\rangle}
\nc{\infint}{\int_{-\infty}^{\infty}}
\nc{\halfwidth}{6.5cm}
\nc{\figwidth}{10cm}
\newcommand{\f}{\frac}

\nc{\nlayers}{L} \nc{\nsectors}{M}
\nc{\indicator}{\mathbf{1}}
\nc{\Rhole}{R_{\rm hole}}
\nc{\Rring}{R_{\rm ring}}
\nc{\neff}{n_{\rm eff}}
\nc{\Frem}{F_{\rm rem}}
\nc{\R}{\mathbb R}
\nc{\C}{\mathbb C}
\nc{\Z}{\mathbb Z}
\nc{\DD}{\Delta}
\nc{\cD}{\mathcal D}
\nc{\lnorm}{\left\|}
\nc{\rnorm}{\right\|}
\nc{\rnormp}{\right\|_{\ell^{p,\eps}}}
\nc{\rar}{\rightarrow}
\begin{document}

\begin{abstract}

	We investigate $L^1\to L^\infty$ dispersive estimates for the three dimensional Dirac  equation with a potential. We also classify the structure of obstructions at the thresholds of the essential spectrum as being composed of a two dimensional space of resonances and finitely many eigenfunctions.  We show that, as in the case of the Schr\"odinger evolution, the presence of a threshold obstruction generically leads to a loss of the natural $t^{-\f32}$ decay rate.  In this case we show that the solution operator is composed of a finite rank operator that decays at the rate   $t^{-\f12}$ plus a term that decays at the  rate  $t^{-\f32}$.
	
\end{abstract}

\title[Dispersive estimates for Dirac Operators]{\textit{Dispersive estimates for Dirac Operators in dimension three with 
obstructions at threshold energies}}

\author[Erdogan, Green, Toprak]{M. Burak Erdo\u{g}an, William~R. Green, and Ebru Toprak}
\thanks{The first and third authors are partially supported by NSF grant DMS-1501041.  The second author is supported by  Simons  Foundation  Grant  511825. }
\address{Department of Mathematics \\
University of Illinois \\
1409 W. Green Street\\
Urbana, IL 61801, U.S.A.}
\email{berdogan@math.uiuc.edu}
\address{Department of Mathematics\\
Rose-Hulman Institute of Technology \\
5500 Wabash Ave.\\
Terre Haute, IN 47803, U.S.A.}
\email{green@rose-hulman.edu}
\address{Department of Mathematics \\
University of Illinois \\
1409 W. Green Street\\
Urbana, IL 61801, U.S.A.}
\email{toprak2@illinois.edu}
%\subjclass{35Q41, 42B20}

\maketitle

\section{Introduction}

We consider the linear Dirac equations in three spatial dimensions with potential,
\begin{align}\label{dirac}
i\partial_t \psi(x,t) = (D_m +V(x)) \psi(x,t), \,\,\,\, \psi(x,0)= \psi_0(x).
\end{align} 
Here $x\in \R^3$ and $\psi(x,t) \in \mathbb{C}^{4}$. The $n$-dimensional free Dirac operator $D_m$ is
 defined by
\begin{align}
   D_m= -i \alpha \cdot \nabla+ m\beta = -i \sum_{k=1}^{n} \alpha_k \partial_k + m\beta,
\end{align}
where $m>0$ is a constant, and(with $N=\lfloor \frac{n+1}{2}\rfloor$, the  $N\times N$ Hermitian matrices $\alpha_j$ satisfy 
\begin{align} \label{matrixeq}
	\left\{\begin{array}{ll}
		\alpha_j \alpha_k+\alpha_k\alpha_j =2\delta_{jk}
		\mathbbm 1_{\mathbb C^{2^{N}}}
		& j,k \in\{1,2,\dots, n\}\\
		\alpha_j \beta+\beta \alpha_j= 
		\mathbbm O_{\mathbb C^{2^{N}}}\\
		\beta^2 = \mathbbm 1_{\mathbb C^{2^{N}}}
		\end{array}
	\right.
\end{align}
Physically, $m$ represents the mass of the quantum particle.  If $m=0$ the particle is massless and if $m>0$ the particle is massive.  We note that dimensions $n=2,3$ are of particular physical importance.
In dimension three we use
\begin{align*}
	\beta=\left[\begin{array}{cc} I_{\mathbb{C}^2} & 0\\ 0 & -I_{\mathbb{C}^2}
	\end{array}\right], \ \alpha_i=\left[\begin{array}{cc} 0 & \sigma_i \\ \sigma_i & 0
	\end{array}\right],
	\end{align*}
	\begin{align*}
	\sigma_1=\left[\begin{array}{cc} 0 & -i\\ i & 0
	\end{array}\right],\
	\sigma_2=\left[\begin{array}{cc} 0 & 1\\ 1 & 0
	\end{array}\right],\
	 \sigma_3=\left[\begin{array}{cc} 1 & 0\\ 0 & -1
	\end{array}\right].
\end{align*}
The Dirac equations \eqref{dirac} were derived   by Dirac as an attempt to tie together the theories of relativity and quantum mechanics to describe quantum particles moving at relativistic speeds.  The relativistic notion of energy, $E=\sqrt{c^2p^2+m^2c^2}$, depends on the particle's mass, momentum and the speed of light.  By combining this with the quantum mechanical notions of energy and momentum $E=i\hbar \partial_t$, $p=-i\hbar \nabla$ one arrives at a non-local equation
\be\label{eq:nonlocal}
	i\hbar \psi_t(x,t)=\sqrt{-c^2\hbar^2 \Delta+m^2c^4 }\, \psi(x,t).
\ee
We note that this is formally the square root of a Klein-Gordon equation. 
Dirac's insight was to linearize this equation into a system of four first order equations.  This linearization leads to the free Dirac equation, \eqref{dirac} with $V\equiv 0$, which describes the evolution of a system of spin up and spin down free electrons and positrons at relativistic speeds.  This systemization allows for the study of a first-order evolution equation, in agreement with a quantum mechanical viewpoint.  In addition, the linearization allows for the incorporation of external electric or magnetic fields in a relativistically invariant manner, which \eqref{eq:nonlocal} or a Klein-Gordon equation cannot.  Another benefit of this system is to account for the spin of the quantum particles.  This interpretation is not without its drawbacks, we refer the reader to the excellent text \cite{Thaller} for a more detailed introduction.

 The linearization, \eqref{dirac}, retains an important property of \eqref{eq:nonlocal} in that the free Dirac operator squared generates a diagonal system of Klein-Gordon equations.
This motivates the following
relationship, which follows from the relationships in \eqref{matrixeq},
\be\label{eq:Dmpm} (D_m -\lambda) (D_m+ \lambda) = (-i \alpha \cdot \nabla + m\beta - \lambda I) (-i \alpha \cdot \nabla + m\beta + \lambda I) = -\Delta + m^2 -\lambda^2 .
\ee
Here the last line is to be interpreted as a diagonal $ 4\times4$ matrix operator. This allows us to
formally define the free Dirac resolvent operator $\mathcal{R}_0(\lambda) = (D_m -\lambda)^{-1}$  in terms of the
free resolvent $R_0(\lambda)= (-\Delta-\lambda)^{-1}$ of the Schr\"odinger operator. That is,
\begin{align}\label{eq:RDR}
\mathcal{R}_0(\lambda) = (D_m+\lambda) R_0(\lambda^2- m^2).
\end{align}

Throughout the paper, we use the notation $X$ to describe a Banach space $X$ and the Banach spaces of $\C^4$ valued functions with components in $X$.
Let $ H^{1}(\R^3) $ be  the first order Sobolev space of the $\C^4$-valued functions,  $f(x) = (f_i(x))_{i=1}^4$, of the spatial variable $x=(x_1,x_2,x_3)$. Then, the free Dirac operator is essentially self-adjoint on $ H^{1}(\R^3) $, its spectrum is purely absolutely continuous and equal to $\sigma_{ess}(D_m)=\sigma_{ac}(D_m)= (-\infty, -m] \cup [m, \infty) $, \cite[Theorem 1.1]{Thaller}.  
Under mild assumptions on $V$, $H:=D_m+V$ is   self-adjoint, and  $\sigma_{ess}(H)= (-\infty, -m] \cup [m, \infty)$, \cite[Theorem 4.7]{Thaller}.

In this paper we aim to study the dispersive bounds by considering the formal solution operator $e^{-itH}$ as an element of the functional calculus via the Stone's formula:
\be\label{eq:Stone}
	e^{-itH}P_{ac}(H)f(x)=\frac{1}{2\pi i}\int_{(- \infty,-m]\cup[m,\infty)} e^{-it\lambda}\big[\mathcal R_V^+-\mathcal R_V^-\big](\lambda)f(x)\, d\lambda,
\ee  
where the perturbed resolvents are defined by $\mathcal R_V^\pm(\lambda)=\lim_{\epsilon \rightarrow 0^+} (D_m+V-(\lambda \pm i\epsilon))^{-1}$. These resolvent operators are well defined as operators between weighted $L^2(\R^3)$ spaces, \cite{agmon,BaHe}. In particular,  
 in \cite[Remark~1.1 and Theorem~3.9]{BaHe},   it was shown that this limit is well-defined as an operator from $ H^{0,s}(\R^3) $ to $ H^{1,-s}(\R^3) $   for any $\lambda \in (-\infty,-m) \cup (m,\infty)\ \setminus \sigma_p(H)$ and $s>\f12$  for  a class of potentials  including those which  we consider in Theorems~\ref{thm:main1} and~\ref{thm:main2}.  Furthermore, for the class of potentials we consider, there are no embedded eigenvalues in the essential spectrum, except
possibly at the thresholds $\lambda=\pm m$, \cite{Thomp}.  See also \cite{Roze,BG1,Vog,GM,BC1}. 

It is known that the Dirac operators can have infinitely many eigenvalues in the spectral gap, see for example \cite{Thaller}.  However, the work of Cojuhari \cite[Theorem 2.1]{Co} guarantees only finitely many eigenvalues in the spectral gap for the class of potentials we consider; also see Kurbenin \cite{Kurb}.  In fact, this result may be obtained as a corollary of our resolvent expansions as in \cite[Remark 4.7]{egd}.

To discuss our main results, we briefly discuss the notion of threshold resonances and eigenvalues.  We characterize both in terms of distributional solutions to the equation
$$
	H\psi= m \psi.
$$
If $\psi \in  L^2(\R^3) $, we say that there is a threshold eigenvalue at $\lambda= m$.  If $\psi \notin  L^2(\R^3) $, but $\la x\ra^{-\f12-\epsilon} \psi \in  L^2(\R^3) $ for all $\epsilon>0$, we say that there is a threshold resonance at $\lambda= m$.  An analagous characterization holds at the threshold $\lambda=-m$.  We provide a detailed characterization of the threshold in Section~\ref{sec:esa}.  If there is neither a threshold resonance or eigenvalue, we say that the threshold is regular.

We take $\chi\in C_c^\infty(\R)$ to be a smooth, even cut-off function of a small neighborhood of the threshold.  That is, $\chi(\lambda)=1$ if $|\lambda-m|<\lambda_0$ for a sufficiently small constant $\lambda_0>0$, and $\chi(\lambda)=1$ if $|\lambda-m|>2\lambda_0$.  
For the duration of the paper, we employ the following notation.   We write $|V(x)|\les \la x\ra^{-\beta}$ to indicate that each component of the matrix $V$ satisfies the bound $|V_{ij}(x)|\les \la x\ra^{-\beta}$.
Our main results are the following low-energy dispersive bounds.

\begin{theorem}\label{thm:main1}

	Assume that   $V$ is a Hermitian matrix for which $|V(x)|\les \la x\ra^{-\beta}$ for some $\beta>7$.  Further, assume that there is a threshold resonance but not an eigenvalue.  Then, there is a time dependent  operator $K_t$, with  rank at most two and satisfying $\sup_t \|K_t\|_{L^1\to L^\infty}\les 1$, such that
	$$
		\left\| e^{-itH}P_{ac}(H)\chi(H)-\la t\ra^{-\f12} K_t 
		\right\|_{L^1\to L^\infty} \les   \la t\ra^{-\f32}.
	$$
\end{theorem}
In fact,  the operator $  K_t$ in the statement can be written as  $K_t= e^{-imt}P_r+\widetilde K_t$ where $P_r$ is a map onto the   threshold resonance space (see Proposition~\ref{prop:res1} below) and $\widetilde K_t$ is a finite rank operator satisfying the family of weighted bounds $\|\la x\ra^{-j}\widetilde{K}_t(x,y) \la y\ra^{-j}\|_{L^{1 }\to L^\infty } \les \la t\ra^{ -j}$ for any $0\leq j\leq 1$.

\begin{theorem}\label{thm:main2}

	Assume that  $V$ is a Hermitian matrix for which $|V(x)|\les \la x\ra^{-\beta}$ for some $\beta>11$.  Further, assume that there is a threshold eigenvalue, then, there is a time dependent, finite rank  operator $K_t$ satisfying $\sup_t \|K_t\|_{L^1\to L^\infty}\les 1$, such that
	$$
		\left\| e^{-itH}P_{ac}(H)\chi(H)-\la t\ra^{-\f12} K_t 
		\right\|_{L^1\to L^\infty} \les \la t\ra^{-\f32}.
	$$
\end{theorem}

This theorem is valid regardless of the existence or non-existence of threshold resonances.  The dynamical, time-decay estimates that we prove provide a valuable contrast to the $L^2$-based conservation laws.  Using these estimates in concert, one can arrive at many other bounds such as Strichartz estimates for the evolution.  Such estimates are often of use when linearizations about special solutions have threshold phenomena for other dispersive equations.  

%In light of \eqref{eq:RDR}, we have
%$$
%\mR_0^{\pm}(\lambda)=(D_m+\lambda)R_0^{\pm}(z^2),\,\,\,\,\,\,\lambda=\sqrt{m^2+z^2}\in (m,\infty). 
%$$

The mathematical analysis of Dirac operators is considerably smaller than the analysis of related equations such as the wave equation, Klein-Gordon or Schr\"odinger equation.  All of the results on three-dimensional Dirac equations in the literature assume that the threshold energies are regular.  The first paper that analyzed the time-decay for a perturbed (massless) Dirac equation was \cite{DF}.  In this paper D'Ancona and Fanelli proved a time-decay rate of $t^{-1}$ for large $t$ for the Dirac equation and related magnetic wave equations provided the potential satisfies a certain smallness condition.  Escobedo and Vega, \cite{EV} provided dispersive and Strichartz estimates for a free Dirac equation in service of analyzing a semi-linear Dirac equation.  In \cite{Bouss1}, Boussaid proved a variety of dispersive estimates for three dimensional Dirac equations.  These estimates were in both the weighted $L^2$ setting and in the sense of Besov spaces.  In this paper it was shown that one can obtain faster decay for large $t$ and smaller singularity as $t\to 0$ provided the initial data is smoother in the Besov sense.  We rely on the high-energy estimates in \cite{Bouss1} to contain our analysis to only a small neighborhood of the threshold.  The high-energy portion of the evolution requires smoothness on the initial data and potential, which we do not need for our results.    To be precise, by taking $p=1$ from Boussaid's general Besov space result, we see

\begin{theorem}[\cite{Bouss1}, Theorem~1.2]
	
	Assume that $V$ is a self-adjoint, $C^\infty$ function that satisfies $|\partial^\alpha V(x)|\les \la x\ra^{\rho+\alpha}$ for some $\rho>5$.  Then, for any $q\in [1,\infty]$, 
	$\theta \in[0,1]$ with $s-s'\geq 2+\theta$, we have
	$$
		\left\| e^{-itH}P_{ac}(H)(1-\chi(H)) \right\|_{B^s_{1,q} \to B^{s'}_{\infty, q}} \les 
		\left\{ \begin{array}{ll}
			|t|^{-1+\f \theta 2} & 0<|t|\leq 1\\
			|t|^{-1-\f \theta 2} & |t|\geq 1
		\end{array}
		\right..
	$$
	
\end{theorem}

If we take $q=1$, and $s'=0$, this gives us a $t^{-\f32}$ decay of the $L^\infty$ norm of the solution, provided the initial data has two derivatives in $L^1$ in the Besov sense.

Our approach relies on a detailed analysis of the Dirac resolvent operators.  We follow the strategy employed 
by the first two authors in \cite{egd} analyzing the two-dimensional Dirac equation with potential, which has roots in the analysis of the two-dimensional Schr\"odinger equation by Schlag \cite{Sc2} and the authors \cite{eg2,eg3}.  In the same manner we build off the work of the first author and Schlag, \cite{ES,ES2}, in which dispersive estimates for the three-dimensional Schr\"odinger operators were studied with threshold resonances and/or eigenvalue.  These results have been sharpened, in terms of assumed decay on the potential, by Beceanu \cite{Bec}.  We note that extending these results on the Schr\"odinger evolution is  non-trivial even for the wave equation, see \cite{KS}.

In addition to proving time decay estimates for the Dirac evolution, we provide a full classification of the obstructions that can occur at the threshold of the essential spectrum at $\lambda =\pm m$.  For the Schr\"odinger equation in three dimensions, there can be a one dimensional space of resonances and/or finitely many eigenfunctions at the threshold.  This classification is inspired by the previous work on Schr\"odinger operators \cite{JN,ES,eg2}, though the rich structure of the Dirac operators provides additional technical challenges. 

Further study of the Dirac operator in the sense of smoothing and Strichartz estimates has been performed by a variety of authors, see for example \cite{BDF,C,CS}.  In the two-dimensional case,   the evolution on weighted $L^2$ spaces was studied in \cite{kopy}, which had roots in the work of Murata, \cite{Mur}.   Frequency-localized endpoint Strichartz estimates for the free Dirac equation are obtained in two and three spatial dimensions in \cite{BH3,BH}, which are used to study the cubic non-linear Dirac equation. Dispersive estimates for a one-dimensional Dirac equation were considered in \cite{CTS}.  During the review period for this article, the first two authors and Goldberg established Strichartz estimates and a limiting absorption principle for Dirac operators in dimension $n\geq2$, \cite{egg}.

In the paper we use the following notations.
The weighted $L^2$ space $L^{2,\sigma}(\R^3)=\{ f\,:\,  \la \cdot \ra^{\sigma}f(\cdot)\in  L^2(\R^3)\}$.  We also write $a-:=a-\epsilon$ for an arbitarily small, but fixed $\epsilon>0$.  Similarly, $a+:=a+\epsilon$.

The paper is organized as follows.
We begin in Section~\ref{sec:exp} by developing expansions for the Dirac resolvent operators.  In Section~\ref{sec:disp} we prove the dispersive bounds in all cases by reducing the bounds to oscillatory integral estimates.  Finally in Section~\ref{sec:esa} we provide a characterization of the threshold resonances and eigenfunctions.

\section{Resolvent expansions around threshold} \label{sec:exp} 
In this section we obtain expansions for the resolvent operators $\mathcal{R}^{\pm}_V(\lambda) $ in a neighborhood of the threshold energies $\pm m $. It is well-known (see e.g. \cite{GS}) that the resolvent, $ R_0^{\pm}(z^2)$, of the free Schr\"odinger operator is an integral operator with kernel    
\begin{align} \label{free}
R_0^{\pm}(z^2)=\frac{ e^{\pm iz |x-y|}} {4\pi |x-y|}= \sum_{j=0}^{\infty} (\pm i z)^j G_j, \,\,\,\,\text{ 
where}
\end{align} 
\begin{align} \label{def:g}
    G_j(x,y) = \frac{1}{4 \pi j!} |x-y| ^{j-1} \,\,\,\ j=0,1,2,...,.
\end{align}
Here we review some estimates (see e.g. \cite{GS,ES}) for $R_0^\pm(z^2)$ needed to study the Dirac evolution.  To best utilize these expansions, we employ the notation
$$
f(z)=\widetilde O(g(z))
$$
to denote
$$
\frac{d^j}{dz^j} f = O\big(\frac{d^j}{dz^j} g\big),\,\,\,\,\,j=0,1,2,3,...
$$
The notation refers to derivatives with respect to the spectral variable $z$, or $|x-y|$ in the expansions for
the integral kernel of the free resolvent operator, which depends on the variable $\rho=z|x-y|$.
If the derivative bounds hold only for the first $k$ derivatives we  write $f=\widetilde O_k (g)$.  In addition, if we write $f=\widetilde O_k(1)$, we mean that differentiation up to order $k$ is comparable to division by $z$ and/or $|x-y|$ as appropriate.
This notation applies to operators as well as
scalar functions; the meaning should be clear from the
context.
 
 In the following analysis we will obtain the expansion on the positive portion $[m, \infty)$ of the spectrum of $H$. A similar analysis with minor changes can be performed to obtain an expansion for the negative portion $(-\infty,-m]$, see Remark~\ref{rmk:negative}.
 
 Writing $ \lambda=\sqrt{m^2+z^2}$ for  $0 <z \ll 1$, and using \eqref{eq:RDR}, we have 
 \begin{multline}
 \label{resolventex}
      \mathcal{R}_0^\pm(\lambda)= \big[ -i \alpha \cdot \nabla + m \beta + \sqrt{m^2+ z^2} I \big] R_0^\pm(z^2) = \\ 
      \big[ -i \alpha \cdot \nabla + m(\beta+I)+ \frac{z^2}{2m} I + O (z^4) I \big] R_0^\pm(z^2) . 
  \end{multline}
For convenience we define  $M_{uc}$ and $M_{lc}$ to be $4\times 4$ matrix-valued operators with kernels  
 \begin{align*}
 M_{uc}= \left[\begin{array}{cc} I_{2\times2} &  0\\ 0 & 0 \end{array}\right] ,   \,\,\,  M_{lc}= \left[\begin{array}{cc} 0 & 0\\ 0 & I_{2\times2} \end{array}\right]. 
 \end{align*} 
 We also have 
the following projections $I_{uc}=\frac12(\beta+I)$ and $I_{lc}=\f12 (I-\beta)$, by 
 $$
I_{uc} \left[\begin{array}{cc} a \\ b \\ c \\ d \end{array}\right] = \left[\begin{array}{cc} a \\ b \\ 0 \\ 0 \end{array}\right], \,\,\,  I_{lc} \left[\begin{array}{cc} a \\ b \\ c \\ d \end{array}\right] = \left[\begin{array}{cc} 0 \\ 0 \\ c \\ d \end{array}\right].
 $$ 
In our expansions we will consider only the `+' case due to the simple relationship between the resolvents $\mathcal R_0^\pm(\lambda)$.
  \begin{lemma} \label{lem:R0exp} Let $r:= |x-y|$, $\lambda=\sqrt{z^2+m^2},$  $0 < z <1$. We have the following expansions for the free resolvent
   \begin{align} \label{eq:R0exp_00}
  \mathcal{R}^+_0(\lambda) &= \mathcal{G}_0 +O\big(z(1+r^{-1} )\big),\\ 
 \label{eq:R0exp_0}   &= \mathcal{G}_0 + iz \mathcal{G}_1 +\widetilde  O_2 \big(z^2  r+z^2 r^{-1} \big),\\
 \label{eq:R0exp_1}  &= \mathcal{G}_0 + iz \mathcal{G}_1 - z^2 \mathcal{G}_2+\widetilde  O_2 \big(z^3 r^{2}+z^3 r^{-1}\big),\\  
 \label{eq:R0exp_J}  &= \sum_{j=0}^J (iz)^j \mathcal{G}_j + \widetilde  O_2 \big(z^{J+\ell}  r^{J+\ell-1}+z^{J+\ell} r^{-1} \big),\,\,J\geq 3, 
  \end{align} 
  for any $0\leq \ell \leq 1$, where
  \begin{align}
 \mathcal{G}_0 (x,y) &= (D_m+mI)G_0(x,y)= [ -i \alpha \cdot \nabla + 2m I_{uc} ] G_0(x,y)=\frac{i\alpha \cdot (x-y)}{4\pi|x-y|^3} + \frac{mI_{uc}}{2\pi |x-y|},
\label{eq:G0 def}  \\
 \mathcal{G}_1(x,y)& = \frac{  m}{2 \pi}M_{uc}(x,y), \label{eq:G1 def} \\
  \mathcal{G}_2(x,y)&=  [ - i \alpha \cdot \nabla +2m I_{uc} ] G_2(x,y) - \frac{1}{2m} G_0 (x,y)\label{eq:G2 def}, 
\\
 	\mathcal G_j(x,y)&= O(\la x-y\ra^{j-1}),\,\,\,j\geq 3.  \nn
 \end{align}
 \end{lemma}
 \begin{proof}
We will only prove \eqref{eq:R0exp_00} and \eqref{eq:R0exp_J} when $J=3$. The proof of the other expansions and the case $J>3$ are similar. 
First using \eqref{free} we have
$$
R_0^{+}(z^2)=\frac{ e^{\pm iz |x-y|}} {4\pi |x-y|}= G_0+  O (z ),  \text{ and}
$$ 
$$
\nabla R_0^{+}(z^2) =\nabla G_0  +  O (z r^{-1}).  
$$
The expansion \eqref{eq:R0exp_00} follows immediately.

To obtain \eqref{eq:R0exp_J} when $J=3$, again using \eqref{free} we have
$$
R_0^{+}(z^2)=\frac{ e^{\pm iz |x-y|}} {4\pi |x-y|}= G_0+izG_1-z^2G_2-iz^3G_3+\widetilde O_2(z^{3+\ell}r^{2+\ell}), \,\,\,0\leq \ell\leq 1,
$$ 
$$
\nabla R_0^{+}(z^2) =\nabla G_0  -z^2\nabla G_2-iz^3\nabla G_3+\widetilde O_2(z^{3+\ell}r^{1+\ell}), \,\,\,0\leq \ell\leq 1.
$$
Using this in \eqref{resolventex}, we have 
\begin{multline*}
\mathcal{R}_0^+(\lambda)  
    =  -i \alpha \cdot [ \nabla G_0  -z^2\nabla G_2-iz^3\nabla G_3]   +  2mI_{uc} (G_0+izG_1-z^2G_2-iz^3G_3) \\ + \frac{z^2}{2m} (G_0+izG_1)   +      \widetilde O_2 \big(  z^{3+\ell} r^{1+\ell} +z^{3+\ell} r^{2+\ell}+z^4r^{-1} \big).
  \end{multline*}
  Note that (for $0<z<1$ and $0\leq \ell\leq 1$)
  $$  \widetilde O_2 \big(  z^{3+\ell} r^{1+\ell} +z^{3+\ell} r^{2+\ell}+z^4r^{-1} \big) =  \widetilde O_2 \big(  z^{3+\ell} r^{2+\ell} +z^{3+\ell} r^{-1}  \big). $$
  Collecting the terms with same $z$ power, and noting that 
  $$  |\nabla G_3|+|G_3|+|G_1|\les \la  x-y \ra^2  $$
  yields   the claim.
 \end{proof}

 To obtain expansions for $\mathcal{R}^{\pm}_V(\lambda)=(D_m +V - (\lambda\pm i 0))^{-1} $ where $ \lambda= \sqrt {z^2 + m^2}$  we utilize the symmetric resolvent identity. First note that, since $ V : \R^3 \rightarrow \C^{4\times 4}$ is self-adjoint, we can write 
   $$ V= B^{*} \Lambda B= B^{*}  |\Lambda|^{\f12} U |\Lambda|^{\f12} B =: v^{*} U v,\,\,\,\text{ where}
   $$
 $$\Lambda=\text{diag}(\lambda_1,\lambda_2,\lambda_3,\lambda_4), \text{ with }\lambda_j \in \R ,$$ 
$$|\Lambda|^{\f12}= \text{diag}(|\lambda_1|^{\f12},|\lambda_2|^{\f12},|\lambda_3|^{\f12},|\lambda_4|^{\f12}),$$
$$ U= \text{diag}(\text{sign}(\lambda_1), \text{sign}(\lambda_2),\text{sign}(\lambda_3),\text{sign}(\lambda_4)).
$$
Defining $ A^\pm(z) = U + v\mathcal{R}^\pm_0(\sqrt {z^2 + m^2}) v^{*}$, as in \cite{egd}, the symmetric resolvent identity yields
\begin{align}\label{eq:pert1}
\mathcal{R}^{\pm}_V(\lambda) =\mathcal{R}^{\pm}_0(\lambda) - \mathcal{R}^{\pm}_0(\lambda) v^{*} (A^{\pm})^{-1} (z) v \mathcal{R}^{\pm}_0(\lambda). 
 \end{align} 
    Note that the statements of Theorems~\ref{thm:main1} and \ref{thm:main2} control operators $L^1(\R^3)$ to $ L^{\infty}(\R^3)$, while in our analysis we invert $A^\pm(z)$ in the $L^2(\R^3)$ setting. Since the leading term of the integral kernel of $\mathcal{R}_0^\pm(\lambda)$ has size $|x-y|^{-2}$, see \eqref{eq:G0 def}, it does not map $L^1(\R^3)$ to $L_{loc}^2 (\R^3)$. However, Remark~\ref{rmk:potential} below shows us the  iterated resolvents provide a bounded map between these spaces. Therefore to use the symmetric resolvent identity, we need two resolvents on both sides of $(A^{\pm})^{-1} (z)$. Accordingly we have 
  \begin{align*}
 \mathcal{R}_V^{\pm}(\lambda) = \mathcal{R}_0^{\pm}(\lambda) - \mathcal{R}_0^{\pm}(\lambda)V\mathcal{R}_0^{\pm}(\lambda)+ \mathcal{R}_0^{\pm}(\lambda)V\mathcal{R}_V^{\pm}(\lambda)V\mathcal{R}_0^{\pm}(\lambda). 
 \end{align*}
 Combining this with $\eqref{eq:pert1}$, we have the identity
\begin{multline}
\label{eq:pert2}
     \mathcal{R}_V^{\pm}(\lambda) = \mathcal{R}_0^{\pm}(\lambda)- \mathcal{R}_0^{\pm}(\lambda)V\mathcal{R}_0^{\pm}(\lambda)+\mathcal{R}_0^{\pm}(\lambda)V\mathcal{R}_0^{\pm}(\lambda)V\mathcal{R}_0^{\pm}(\lambda) \\ 
	+\mathcal{R}_0^{\pm}(\lambda)V\mathcal{R}_0^{\pm}(\lambda)v^*(A^{\pm})^{-1} (z) v \mathcal{R}_0^{\pm}(\lambda)V\mathcal{R}_0^{\pm}(\lambda). 
\end{multline}

\begin{lemma} \label{lem:potential} Let $|V(x)| \les \la x \ra ^{-\beta}$ where $\beta>2$, let $1\leq l,k<3$, with $l+k <\f92$   and $\sigma>\f12$. Then we have 
$$ \sup _{x  \in \R^3}\Big\| \int\frac{1}{|x-x_1|^l} |V(x_1)| \frac{1}{|y-x_1|^k} dx_1\Big\| _{L_y^{2,-\sigma} } \les 1. $$
The conclusion remains valid in the case $l$ or  $k$ is zero, provided $l+k< 3$, $\beta>3$ and $\sigma>\f32$.
\end{lemma}
For the proof of Lemma~\ref{lem:potential}, we use the following  lemma  from \cite{EG1}.
\begin{lemma} \label{lem:potential2}Fix $u_1$, $u_2 \in \R^n$ and let $ 0 \leq k$, $l <n$ , $\beta >0$ , $k+l+\beta \geq n$ , $k+l \neq n $. We have
        $$ \int_{\R^n} \frac{\la x \ra ^{-\beta-}}{ |x- u_1|^k |x-u_2|^l} dx \les \left\{
            \begin{array}{ll}
            (\frac{1}{|u_1-u_2|})^{\max(0, k+l-n)} & \quad  |u_1-u_2| \leq 1, \\
             (\frac{1}{|u_1-u_2|})^{\min(k, l, k+l+\beta -n)} & \quad |u_1-u_2| > 1.
        \end{array} \right .
        $$
\end{lemma}   
  \begin{proof}[Proof of Lemma~\ref{lem:potential}]
Using \ref{lem:potential2} we can obtain the following bound when $l,k\geq 1$ and $l+k<\f92$.
$$ 
\int_{\R^3} \frac{\la x_1 \ra ^{-\beta-}}{ |x- x_1|^k |x_1-y|^l} dx_1 \les \frac{1}{|x-y|} +\frac{1}{|x-y|^{\f32-}}
$$
 provided $\beta > 2$. 
 Note that when $k+l=3$ we can apply the lemma after using the inequality
 $$
 \frac{1}{ab^2} \les \frac{1}{ab^{2-}}+ \frac{1}{ab^{2+}}  \,\,\ \text{for any} \ a,b >0.
 $$
This yields the first part of the lemma since for $\sigma>\frac12$ we have    
$$
     \sup_{x\in \R^3}\left\| \frac{\la y \ra^{ -\sigma}}{|x-y|^{\f32-}} \right\|_{L^2_y(\R^3)}, \quad
     \sup_{x\in \R^3}
      \left\| \frac{\la y \ra^{ -\sigma}}{|x-y|} \right\|_{L^2_y(\R^3)}  \les 1.
$$
If at least one of $l,k=0$ then we pick $\beta>3$ so that 
$$ \sup_{x\in \R^3} \int_{\R^n} \frac{\la x_1 \ra ^{-\beta-}}{ |x- x_1|^k |x_1-y|^l} dx_1 \les 1\in L^{2,-\f32-}_y(\R^3).$$
%Note that $ \la y \ra ^{-\alpha} \in L^{2,-\sigma}$ provided $ \sigma+\alpha > {\f 32}$. 
\end{proof}

\begin{rmk} \label{rmk:potential}Using Lemma~\ref{lem:potential} one can conclude that for any $|V(x)| \les \la x \ra ^{-2-}$ and $\sigma> {\f 12}$, 
  $$ \sup _{x  \in \R^3} \| \mathcal{R}_0^{\pm}(\lambda) V \mathcal{R}_0^{\pm} (\lambda) \| _{L_y^{2,-\sigma} } \les \la z \ra^2. $$
Indeed, using  \eqref{resolventex}, we have
 $$ | \mathcal{R}_0(\lambda) | \les \frac{1}{|x-x_1|^2}  + \frac{   \la z\ra}{ |x-x_1| }$$
and accordingly,
$$
| \mathcal{R}_0(\lambda)(x,x_1)V(x_1)\mathcal{R}_0(\lambda)(x_1,y)| \les  \la z\ra^2  \sum_{l,k\in\{1,2\}} \frac{\la x_1 \ra ^{-2-}}{|x-x_1|^k|y-x_1|^l}.
$$
This gives the claim by Lemma~\ref{lem:potential}. 

  \end{rmk}

\begin{defin} \label{def:HS} We say that an operator $T(z)$ with kernel $T(x,y)$ is absolutely bounded if $|T(x,y)|$ gives rise to a bounded operator from $L^2(\R^3)$ to $L^2(\R^3)$. We use the representation $T(z) = \widetilde{O}_j(z^p) $ if $T(z)$ satisfies the bounds 
$ \| |\partial_z^k T(z)| \|_{L^2 \rightarrow L^2} \les z^{p-k}$ for $k=1,2,3,...,j$.
\end{defin}
\begin{defin} An operator $T $ is Hilbert-Schmidt if its kernel $T(x,y)$ satisfies 
$$
\| T\|_{HS}^2=  \int _{\R^n} \int _{\R^n}  |T(x,y)|^2 dx dy   < \infty.
$$
\end{defin}
Hilbert-Schmidt operators and finite rank operators are absolutely bounded. 

 We have developed expansions for $\mathcal R_0^+(\lambda )$ using the Schr\"odinger resolvent $R_0^+(z^2)$. We  develop expansions for $A(z):=A^+(z)$ when $z>0$ and $A(z):=A^-(-z)$ when $z<0$. It follows from from \eqref{free} that  $A^-(z)=A^+(-z)$.
 \begin{lemma} \label{lem:M0M1}
 Let $|V(x)| \les \la x \ra ^{-\beta}$ for some $\beta>0$, and define $A_0:=U + v\mathcal{G}_0v^*$. Then we have the following expansions for $A(z)$ when $|z|<1$.  
\begin{align*}
A(z) &= A_0 + iz v\mathcal{G}_1v^* - z^2 v\mathcal{G}_2 v^* -i z^3  v\mathcal{G}_3v^* + M_0(z), \\
       &= A_0 + iz v\mathcal{G}_1v^* - z^2 v\mathcal{G}_2 v^* -i  z^3  v\mathcal{G}_3v^* + z^4 v\mathcal{G}_4 v^* + iz^5 v\mathcal{G}_5 v^*+ M_1(z),\,\,\,\,\,\text{ where}\\
  & M_0(z)=\widetilde{O}_2(z^{3+})  \,\ \text{if} \,\  \beta > 7 \,\ \text{and} \,\, M_1(z)=\widetilde{O}_2(z^{5+}) \,\,\  \text{if} \,\  \beta > 11.
 \end{align*}
\end{lemma} 
\begin{proof} By the Definition~\ref{def:HS} it is enough to show that  $\| \partial_z^k M_0(z)(x,y) \|_{HS} \les z^{(3-k)+}$ and $\| \partial_z^k M_1(z)(x,y) \|_{HS} \les z^{(5-k)+}$ for the given value(s) of $\beta$.
Using the expansion  \eqref{eq:R0exp_J} with $J=3$ and $J=5$ respectively, and $\ell=0+$ we have
\begin{align*}
&|\partial_z^k M_0(z)(x,y)| \les z^{(3-k)+} \Big( \frac{|v(x)||v^*(y)|}{|x-y|} + |v(x)||x-y|^{2+}|v^*(y)| \Big), \\
&|\partial_z^k M_1(z)(x,y)| \les z^{(5-k)+} \Big( \frac{|v(x)||v^*(y)|}{|x-y|} + |v(x)||x-y|^{4+}|v^*(y)| \Big),
\end{align*}
for $k=0,1,2$.  
 $\frac{|v(x)||v^*(y)|}{|x-y|}$ is Hilbert-Schmidt provided $|v(x)| \les \la x \ra ^{-1-}$, and for $p\geq 0$, $|v(x)||x-y|^p|v^*(y)|$ is Hilbert-Schmidt provided $|v(x)| \les \la x \ra^{-p - \frac{3}{2}-}$.
\end{proof}
Lemma~\ref{lem:M0M1} together with Lemma~\ref{lem:inverse} shows that the invertibility of $A(z)$ as an operator on $L^2$ for small $z$ depends upon the invertibility of the operator $A_0$ on $L^2$. Before we discuss the invertibility of  $A(z)$ we give the following definitions for resonances at the threshold $\lambda=m$.
  \begin{defin}\label{def:res defn}
  \begin{enumerate}
  \item We say that $\lambda=m$ is a regular point of the spectrum of $H=D_m+V$ provided  $A_0=U+v\mathcal{G}_0v^*$ is invertible on $L^2(\R^3)$.  
  \item Assume that $\lambda=m$ is not a regular point of the spectrum. Then we define $S_1$ as the Riesz projection onto the kernel of $A_0 $ as an operator on $L^2(\R^3)$. In this case $A_0+S_1$ is invertible. Accordingly we define $D_0:= (A_0+S_1)^{-1}$. We say that there is a resonance of the first kind at the threshold ($\lambda=m$) if $ S_1v\mathcal{G}_1v^*S_1$ is invertible in $S_1L^2$, in this case we define $D_1:= (S_1v\mathcal{G}_1v^*S_1)^{-1}$.
  \item Assume $ S_1v\mathcal{G}_1v^*S_1$ is not invertible. Let $S_2$ be the Riesz projection onto the kernel of $ S_1v\mathcal{G}_1v^*S_1$ as an operator on $S_1L^2(\R^3)$. Then $S_1v\mathcal{G}_1v^*S_1+S_2$ is invertible on $S_1L^2(\R^3)$ and we denote $D_2:=(S_1v\mathcal{G}_1v^*S_1+S_2)^{-1}$.  We say there is a resonance of the second kind at threshold if $S_2=S_1\neq 0$. If $S_2\neq 0$ and $S_2 \neq S_1$, we say there is a resonance of the third kind.
  \end{enumerate}
   \end{defin}
   
\begin{rmk} \label{rmk:negative}
\begin{enumerate}[(i)]
\item We provide a full characterization of the threshold obstructions  and relate them to various spectral subspaces of $H=D_m+V$  in Section~\ref{sec:esa}.
In particular $S_1\neq 0$, $S_1 \neq S_2$ corresponds to the existence of a resonance and $S_2 \neq 0$ corresponds to the existence of an eigenvalue at the threshold. A resonance of the first kind indicates that there is a threshold resonance  but not an eigenvalue. % A resonance of the second kind  indicates that there is an eigenvalue but not a resonance. Finally, a resonance of the third kind indicates that there is both a resonance and an eigenvalue at threshold energy. 

\item Note that $v\mathcal{G}_0v^*$ is compact and self-adjoint. Hence, $A_0$ is a compact perturbation of $U$ and it is self-adjoint. Also, the spectrum of $U$ is in $\{-1,1\}$. Hence, zero is the isolated point of the spectrum of $A_0$ and $dim(Ker_{A_0})$ is finite. Since $S_2\leq S_1$, $S_2$ is also a finite rank projection. In addition, if there is resonance of the first kind then the range of $S_1$ is at most two dimensional, see Corollary~\ref{cor:2d}.  Heuristically, the rank of $S_1$ being at most two corresponds to the possibility of having a `spin up' and a `spin down' resonance function at the threshold energy.
\item We do our analysis in the positive portion of the spectrum $[m,\infty)$ and develop expansions of $\mathcal{R}_V$ around the threshold $\lambda=m$. One can do the same analysis for the negative portion of the spectrum taking $\lambda = -\sqrt{z^2+m^2}$. In this case the perturbed equation has a threshold resonance or eigenvalue at $\lambda=-m$ is related to distributional solutions of  $ (H+ mI )g=0$.

\item We have
$$D_0S_1=S_1D_0=S_1,$$
and similarly for $S_2$ and $D_2$. We prove below that $D_0$ is absolutely bounded. The absolute boundedness of $D_1$, $D_2$ is clear since they are  finite rank operators.
\end{enumerate}
\end{rmk}
\begin{lemma} The operator $D_0$ is absolutely bounded in $L^2(\R^3)$.
\end{lemma}
\begin{proof} Recall that $D_0 = (U+v\mathcal{G}_0v^* + S_1)^{-1}$. Using the resolvent identity 
%$$
%D_0= U - D_0 (v\mathcal{G}_0v^* + S_1) U 
%$$
twice we obtain
\begin{align}\label{resolvD0}
D_0= U - U(v\mathcal{G}_0v^* + S_1) U + D_0(v\mathcal{G}_0v^* + S_1) U(v\mathcal{G}_0v^* + S_1) U.
\end{align} 
 Note that $U$ is absolutely bounded. Also note that since $S_1$ is finite rank, any summand containing $S_1$ is finite rank, and hence absolutely bounded. Using \eqref{eq:G0 def}, we have
$$ | \mathcal{G}_0(x,y)| \leq c_1 I_1(x,y) + c_2 I_2(x,y), $$ 
where $I_1$ and $I_2$ are the fractional integral operators. One can see that these two operators are compact operators on $L^{2,\sigma} \rightarrow L^{2,-\sigma}$ for $\sigma >1$, see Lemma~2.3 in \cite{Jen}. Therefore $v\mathcal{G}_0v^*$ is absolutely bounded.

It remains to prove that  
\be\label{eq:gecici}
D_0 v\mathcal{G}_0v^* U v\mathcal{G}_0v^* U=D_0 v\mathcal{G}_0V\mathcal{G}_0v^*  U
\ee
 is absolutely bounded.  Recalling the definition of $\mathcal{G}_0$ given with \eqref{eq:G0 def} one can see that the operator
 $  v\mathcal{G}_0V\mathcal{G}_0v^* U$ is Hilbert-Schmidt  by Lemma~\ref{lem:potential} for any $|v(x)| \les \la x \ra ^{-2-}$. Finally, being the composition of a bounded operator, $D_0$,  and a Hilbert-Schmidt operator,  $  v\mathcal{G}_0V\mathcal{G}_0v^* U$, \eqref{eq:gecici} is Hilbert-Schmidt and hence absolutely bounded.
\end{proof}
We use the following lemma from \cite{JN} to invert the operator $A(z)=U + v\mathcal{R}_0(\sqrt {z^2 + m^2}) v^{*}$ around $z=0$, ($\lambda=m$).   

  \begin{lemma} \label{lem:inverse}
  
   Let $ \mathbb{F} \subset \C \setminus\{0\}$ have zero as an accumulation point. Let $A(z)$, $z\in \mathbb{F}$, be a family of bounded operators of the form 
   $$ A(z) =  A_0 +z  A_1(z) $$
 with $A_1(z)$ uniformly bounded as $z\rightarrow 0$. Suppose that $z=0$ is an isolated point of the spectrum of $A_0$, and let $S$ be the corresponding Riesz projection. Assume that   rank(S) $<\infty$. Then for sufficiently small $z\in \mathbb{F}$ the operators 
    \begin{align} \label{def:B(z)}
    B(z):= {\f 1 z} ( S- S(A(z)+S)^{-1} S) 
    \end{align}
 are well-defined and bounded on $\mathcal{H}$. Moreover, if $A_0=A_0^{*}$, then they are uniformly bounded as  $z\rightarrow 0$. The operator $A(z)$  has bounded inverse in $\mathcal{H}$ if and only if B(z) has a bounded inverse in $S\mathcal{H}$, and in this case 
   \begin{align} \label{def:Ainv}
   A^{-1} (z) = (A(z)+S)^{-1} + {\f 1 z} (A(z)+S)^{-1}  SB^{-1}(z)S(A(z)+S)^{-1}. 
   \end{align}
\end{lemma}  
\begin{lemma}\label{lem:A+S_1} Suppose that $\lambda=m$ is not a regular point of the spectrum of $H=D_m +V$, with $|V(x)|\les \la x\ra^{-\beta}$ for some $\beta>0$, and let $S_1$ be the Riesz projection from Definition~\ref{def:res defn}. Then for sufficiently small $z_0>0$ , the operator $A(z) +S_1$ is invertible for all $0<|z|<z_0 <1$  as a bounded operator on $L^2(\R^3) \rightarrow  L^2(\R^3)$. Further, one has
\begin{align}  \label{eq:A+S-1}
\begin{split}
(A(z)+S_1)^{-1} &=D_0 - iz [ D_0 v \mathcal{G}_1 v^*D_0]+ z^2 [ D_0v\mathcal{G}_2v^*D_0 - D_0v\mathcal{G}_1v^*D_0 v\mathcal{G}_1v^*D_0 ] \\
                                     &+ z^3 \Gamma_0 + \widetilde{O}_{3}(z^{3+}) \,\,\ \text{for} \,\, \beta>7,
                                    \end{split}
%\Gamma_1=[ - D_0 v\mathcal{G}_3v^* D_0 + D_0v\mathcal{G}_1v^*D_0v\mathcal{G}_2v^*D_0 + D_0v\mathcal{G}_2v^*D_0 v\mathcal{G}_1v^*D_0 - D_0 v \mathcal{G}_1 v^*D_0 v \mathcal{G}_1 v^*D_0v \mathcal{G}_1 v^*D_0 ]
   \end{align}                                   
\begin{align}\label{eq:A+S-2}
\begin{split}
(A(z)+S_1)^{-1} &=D_0 - iz [ D_0 v \mathcal{G}_1 v^*D_0]+ z^2 [ D_0v\mathcal{G}_2v^*D_0 - D_0v\mathcal{G}_1v^*D_0 v\mathcal{G}_1v^*D_0  ] \\
                                     &+ z^3\Gamma_0 + z^4 \Gamma_1 + z^5 \Gamma_2 + \widetilde{O}_5(z^{5+}) \,\, \text{for} \,\, \beta > 11.
                                     \end{split}
                                     \end{align}
                                                                          
%\Gamma_1=[ -D_0 v \mathcal{G}_4 v^*D_0 + D_0v\mathcal{G}_3v^*D_0v\mathcal{G}_3 v^*D_0 + D_0v\mathcal{G}_4v^*D_0v\mathcal{G}_2v^*D_0 + D_0v\mathcal{G}_2v^*D_0v\mathcal{G}_4v^*D_0 \\
                                     % & \hspace{20mm} - D_0v\mathcal{G}_1v^*D_0 v\mathcal{G}_1v^*D_0v\mathcal{G}_2v^*D_0 - D_0v\mathcal{G}_1v^*D_0 v\mathcal{G}_2v^*D_0v\mathcal{G}_1v^*D_0 \\
                                    % & \hspace{20mm} - D_0v\mathcal{G}_2v^*D_0 v\mathcal{G}_1v^*D_0v\mathcal{G}_1v^*D_0 \\
                                    % & \hspace{20mm} +  D_0v\mathcal{G}_1v^*D_0v\mathcal{G}_1v^*D_0v\mathcal{G}_1v^*D_0v\mathcal{G}_1v^*D_0]  
 
 %\Gamma_2 = [  D_0 v \mathcal{G}_5 v^*D_0 + D_0v\mathcal{G}_1v^*D_0v\mathcal{G}_4 v^*D_0 +D_0v\mathcal{G}_4v^*D_0v\mathcal{G}_1 v^*D_0+D_0v\mathcal{G}_2v^*D_0v\mathcal{G}_3 v^*D_0   \\
                                %  &  \hspace{20mm} +D_0v\mathcal{G}_3v^*D_0v\mathcal{G}_2 v^*D_0 + D_0v\mathcal{G}_1v^*D_0v\mathcal{G}_1v^*D_0v\mathcal{G}_2v^*D_0 \\
                               %   & \hspace{20mm} +D_0v\mathcal{G}_1v^*D_0v\mathcal{G}_2v^*D_0v\mathcal{G}_1v^*D_0 +D_0v\mathcal{G}_2v^*D_0v\mathcal{G}_1v^*D_0v\mathcal{G}_1v^*D_0 \\
                                %     & \hspace{20mm} + D_0v\mathcal{G}_1v^*D_0v\mathcal{G}_1v^*D_0v\mathcal{G}_1v^*D_0v\mathcal{G}_2v^*D_0(other comb) \\
                               %      & \hspace{20mm} + D_0v\mathcal{G}_1v^*D_0v\mathcal{G}_1v^*D_0v\mathcal{G}_1v^*D_0v\mathcal{G}_1v^*D_0v\mathcal{G}_1v^*D_0                                                                              
Here $\Gamma_0$, $\Gamma_1$ and $\Gamma_2$ are z independent absolutely bounded operators.                                   
\end{lemma}
\begin{proof} 
We use Neumann series expansion using Lemma~\ref{lem:M0M1}. The operators $\Gamma_0$, $\Gamma_1$ and $\Gamma_2$ are absolutely bounded since they are composition of Hilbert Schmidt operators with absolutely bounded operators. 
\end{proof}
The following lemma gives an expansion for $A^{-1}(z)$ for $0<|z|<z_0$ when there is a resonance of the first kind at threshold energy. 
 \begin{lemma} \label{lem:invexp1} Let $|V(x)| \les \la x \ra ^{-7- }$. If there is a resonance of the first kind at the threshold $\lambda=m$, then 
$$
A^{-1}(z) = -{\f  i z}S_1D_1S_1+ E(z)
$$
where $E(z)$ is an absolutely bounded operator satisfying
$$\Big\| \sup_{|z|<z_0} |\partial_z^{k} E(z)| \Big\|_{L^2 \rightarrow L^2} \les 1$$ for $k=0,1,$ and $\| |\partial_z^{2} E(z)|  \|_{L^2 \rightarrow L^2} \les z^{-1+} $.
\end{lemma}
\begin{proof}
Recall that using Lemma~\ref{lem:inverse} in order to invert $A(z)$ first we need to check the invertibility of 
$$
B(z)={\f 1 z} ( S_1- S_1(A(z)+S_1)^{-1} S_1)
$$
on $S_1L^2$. Noting that $S_1D_0=S_1$ and using \eqref{eq:A+S-1}, we have
\begin{align}\label{exp:b} B(z) = i S_1v\mathcal{G}_1 v^* S_1 - z [S_1v\mathcal{G}_2v^*S_1- S_1v\mathcal{G}_1v^*D_0 v\mathcal{G}_1v^*S_1] + z^2 S_1\Gamma_0 S_1 + \widetilde O_2 ( z^{2+}). 
\end{align}

Recall by Definition~\ref{def:res defn}, if there is a resonance of the first kind then $S_1v\mathcal{G}_1 v^* S_1$ is invertible. Hence, $B(z)$ is invertible and for sufficiently small $z$, we have
\be \label{eq:Binv1}
B^{-1}(z) =- i D_1 + z \Gamma_3 + z^2\Gamma_4 +\widetilde O_2(z^{2+}).
\ee
Note that $\Gamma_i$'s in here are composition of z independent, absolutely bounded operators. The absolute boundedness follows since $S_1$ is finite rank.

Using this expression together with \eqref{eq:A+S-1} in \eqref{def:Ainv}, we have
\begin{align*}
A^{-1}(z) &= (A(z)+S_1)^{-1} + \f 1 {z} (A(z)+S_1)^{-1}S_1 B^{-1} (z) S_1(A(z)+S_1)^{-1} \\
               &=- \f i {z} S_1D_1S_1 + z\Gamma_5+\widetilde O_2(z^{1+})  = - {\f i z} S_1D_1S_1+ E(z).
\end{align*}
The bounds on the operator $ E(z)$ follow from  \eqref{eq:A+S-1} and \eqref{eq:Binv1}. 
\end{proof}

 The following lemma gives  the expansion for $A^{-1}(z)$ in the   cases when there is a resonance of the second or third kind at the threshold, that is when there is a threshold eigenvalue.
 \begin{lemma}\label{lem:generalainverse}Let $|V(x) | \les \la x \ra^{-11-}$. If   there is a resonance of the second or third kind at the threshold $\lambda=m$, then we have 
\begin{align*}
A^{-1}(z) =- \frac{1}{z^2} S_2D_3S_2 + \frac{1}{z} \Omega+ E(z).
\end{align*}
where $S_2D_3S_2$ and  $\Omega$ are finite rank  operators. Furthermore,
\begin{align*}
\| \sup_{|z|<z_0} |\partial_z^{k} E(z)| \|_{L^2 \rightarrow L^2} \les 1, \,\,\ \text{for} \,\,\ k=0,1, \,\, \text{and} \,\, \| |\partial_z^{2} E(z)| \|_{L^2 \rightarrow L^2} \les z^{-1+}. 
\end{align*}

\end{lemma}

\begin{proof}
 Recall that in this case the operator $S_1v\mathcal{G}_1v^*S_1$ is not invertible and we defined  $S_2$ to be the projection on the kernel of $S_1v\mathcal{G}_1v^*S_1$. In the following proof we use Lemma~\ref{lem:inverse} twice; to first invert $B(z)$ and then to invert $A(z)$. 
 
 Noting the  leading term of  \eqref{exp:b}, in order to use the invertibility of $S_2 + S_1  v \mathcal{G}_1 v^* S_1$ we invert $-iB(z)+S_2$ on $S_1L^2$, and use Lemma~\ref{lem:inverse} to invert $-iB(z)$, hence $B(z)$. Using the expansion \eqref{eq:A+S-2} in \eqref{def:B(z)} we have 
\begin{align*}
-iB(z)+S_2= &[S_2 + S_1  v \mathcal{G}_1 v^* S_1] +iz  [S_1v\mathcal{G}_2v^*S_1 - S_1v\mathcal{G}_1v^*D_0 v\mathcal{G}_1v^*S_1  ]  + z^2 \Gamma_6 \\
                                     & + z^3 \Gamma_7+ z^4\Gamma_8+ \widetilde{O}_2(z^{4+}).
\end{align*}
with $\Gamma_i$ absolutely bounded operators independent of $z$. 

We denote $D_2=(S_1v\mathcal{G}_1v^*S_1+S_2)^{-1}$. By Neumann series expansion for small $|z|$ we have 
\begin{multline} \label{eq:b+s_2inv}
(-iB(z)+S_2 )^{-1}  = D_2 -izD_2 [S_1v\mathcal{G}_2v^*S_1- S_1v\mathcal{G}_1v^*D_0 v\mathcal{G}_1v^*S_1 ]  D_2  \\
  + z^2 \Gamma_9 
      + z^3\Gamma_{10} +z^4 \Gamma_{11} +\widetilde{O}_2(z^{4+}),
\end{multline}
where the $\Gamma_i$'s are absolutely bounded operators independent of $z$. 
% \Gamma_4=- D_0v\mathcal{G}_3v^*D_0+ D_0v\mathcal{G}_1v^*D_0v\mathcal{G}_2v^*D_0 + D_0v\mathcal{G}_2v^*D_0 v\mathcal{G}_1v^*D_0 - D_0v\mathcal{G}_1v^*D_0v\mathcal{G}_1v^*D_0v\mathcal{G}_1v^*D_0 
 %\Gamma_5 = [ D_2S_1\Gamma_1S_1 D_2S_1\Gamma_3S_1D_2 S_1\Gamma_1S_1D_2 + D_2S_1\Gamma_2S_1D_2S_1\Gamma_1S_1D_2 \\
      %& \hspace{2cm} + D_2S_1\Gamma_2S_1 D_2S_1\Gamma_1S_1D_2 + D_2S_1\Gamma_3S_1D_2 ]$%.
Then, noting that $S_1S_2=S_2S_1= S_2$, $S_2D_2= D_2S_2=S_2$,
\begin{align*}
B_1(z) &:= \frac{S_2 - S_2(-iB(z) +S_2)^{-1}S_2}{z} \\
 &= iS_2 v\mathcal{G}_2v^*S_2+ S_2v\mathcal{G}_1v^*D_0 v\mathcal{G}_1v^*S_2  +z S_2 \Gamma_9S_2   
+ z^2 S_2 \Gamma_{10} S_2 + z^3 S_2 \Gamma_{11} S_2 + \widetilde{O}_2(z^{3+}) \\
   & =i S_2v\mathcal{G}_2v^*S_2 + z S_2 \Gamma_9S_2 + z^2 S_2 \Gamma_{10} S_2 + z^3 S_2 \Gamma_{11} S_2 + \widetilde{O}_2(z^{3+}).
\end{align*} 
For the third equality we used that   $\mathcal{G}_1v^*S_2=0$, (see Corollary~\ref{cor:S2 orth}).  By Lemma~\ref{lem:esa3}, the operator $S_2v\mathcal{G}_2v^*S_2$ is invertible on $S_2L^2$. Letting $D_3:= (S_2v\mathcal{G}_2v^*S_2 )^{-1}$ we have
\begin{align} \label{eq:B1inv}
B_1(z)^{-1}= - i D_3 + z \Gamma_{12}  + z^2\Gamma_{13} +z^3 \Gamma_{14} + \widetilde{O}_2(z^{3+}).
\end{align}
Here $\Gamma_i$'s are finite rank operators since $S_2$ is finite rank.  Further, they are independent of $z$.

Using this expression in \eqref{def:Ainv} for $(-i B(z))^{-1}=i B^{-1}(z) $, we have 
$$
 B^{-1}(z)=-i (-iB(z)+S_2)^{-1} - {\f {i} z} \Big[ (-i B(z)+S_2)^{-1} S_2(B_1 (z) )^{-1} S_2(-iB(z)+S_2)^{-1} \Big].
$$
Plugging this in \eqref{def:Ainv} we have,
\begin{multline} 
A^{-1} (z)  = (A(z)+S_1)^{-1} - {\f i z} \Big[ (A(z)+S_1)^{-1} S_1(-iB (z)+S_2)^{-1} S_1(A(z)+S_1)^{-1} \Big]  \\
                - {\f i {z^2}}  \Big[ (A(z)+S_1)^{-1} S_1(-iB (z)+S_2)^{-1} S_2 B_1^{-1}(z) S_2 (-iB (z)+S_2)^{-1}S_1(A(z)+S_1)^{-1} \label{exp:Ainvlong} \Big].
\end{multline}
Inserting the expansions \eqref{eq:A+S-2}, \eqref{eq:b+s_2inv}, and \eqref{eq:B1inv} in this equality we obtain
$$
 A(z)^{-1} = - \frac{1}{z^2} S_2D_3S_2 + \frac{1}{z} \Omega +\Omega_0+z\Omega_1 +\widetilde{O}_2(z^{1+})  
               = - \frac{1}{z^2} S_2D_3S_2 + \frac{1}{z} \Omega + E(z).
$$
Here $\Omega_j$'s are absolutely bounded operators independent of $z$. Also, $\Omega$ is a finite rank operator. Note that by \eqref{exp:Ainvlong}, $\Omega$ is the sum of a composition of $z$ independent operators, at least one of which is $S_1$ or $S_2$. The fact that $S_1$ and $S_2$ are finite rank operators establishes the claim.
\end{proof}

\section{Dispersive estimates}\label{sec:disp}
In this section we prove Theorems~\ref{thm:main1} and \ref{thm:main2} through a careful analysis of the oscillatory integrals that naturally arise in the Stone's formula \eqref{eq:Stone}.  We divide this into three subsections.  First, in Subsection~\ref{subsec:Born}, we consider the Born series terms and show that they satisfy the bound $\la t\ra^{-\f32}$ as an operator from $L^1(\R^3)\to L^\infty(\R^3)$.  In Subsections~\ref{subsec:noe} and \ref{subsec:e}, we show that the singular terms that arise in the expansion of the spectral measure when there are threshold resonances or eigenvalues yield a slower time decay rate, but are finite rank operators.

Recall the expansion \eqref{eq:pert2} for the perturbed resolvent.  To emphasize the change of variables and dependence now on the spectral parameter $z$, we write the resolvents as $\mathcal R_0(z)$ rather than $\mathcal R_0(\lambda)$.  Under this identification, we have $\mathcal R_0^-(z)=\mathcal R_0^+(-z)$.  Without loss of generality, we take $t>0$, the proof for $t<0$ requires only minor adjustments. We consider integrals of the form below  for the contribution of the finite terms of the Born series \eqref{eq:pert2} to the Stone's formula \eqref{eq:Stone}.  
$$ \int_m^{\infty} e^{-it\lambda} \chi(\lambda) \Big[\mathcal{R}^{+}_0(z) \big(V\mathcal{R}^{+}_0(z)\big)^k - \mathcal{R}^{-}_0(z) \big(V\mathcal{R}^{-}_0(z)\big)^k \Big] d \lambda. $$
Recall that $\lambda=\sqrt{z^2+m^2}$, we can re-write this as 

\begin{align} \label{finiteint}
 \int_{0}^{\infty} e^{-it \sqrt{z^2+m^2}} \frac{z \chi(z)}{\sqrt{z^2+m^2}} \Big[\mathcal{R}^{+}_0(z) \big(V\mathcal{R}^{+}_0(z)\big)^k - \mathcal{R}^{-}_0(z) \big(V\mathcal{R}^{-}_0(z)\big)^k \Big] dz.
\end{align}
%For the proof of Lemma~\ref{lem:born} we will use the integral formula \eqref{finiteint}. 
%However, we note that since $\mathcal{R}^{-}_0(-z) = \mathcal{R}^{+}_0(z) =: \mathcal{R}_0(z)$ the integral can be extended to the real line as 
%$$\int_{-\infty}^{\infty} e^{-it \sqrt{z^2+m^2}} \frac{z \chi(z)}{\sqrt{z^2+m^2}} \mathcal{R}_0(z) \big(V\mathcal{R}_0(z)\big)^k dz. $$ This version will be useful in the next sections.  
We utilize from the following consequence of the classical Van der Corput lemma, \cite{St}.
\begin{lemma} 
\label{lem:vdc} 
If $\phi:[a,b] \rightarrow \R $ obeys the bound $ | \phi '' (z) | \geq  t >0$ for all $z\in [a,b]$, and if $\psi: [a,b] \rightarrow \mathbb{C}$ such that $\psi^\prime \in L^1([a,b])$, then 
$$
\Bigg| \int_a^b e^{i\phi(z)} \psi(z)\, dz \Bigg| \les t^{-{\f 12}} \Bigg\{ |\psi(b)| + \int_a^b |\psi^\prime(z)|\, dz \Bigg\}.
$$
\end{lemma}

\subsection{The Born Series}\label{subsec:Born}

We have the following lemma for the finite terms of Born series.
\begin{prop} \label{lem:born}
 Let $|V(x)|\les \la x \ra ^{-3-}$. Then for any  $k\in \mathbb N\cup \{0\}$, the following bound holds 
\begin{align}
\sup_{x,y \in \R^3} \Big| \int_{0}^{\infty} e^{-it \sqrt{z^2+m^2}} \frac{z \chi(z)}{\sqrt{z^2+m^2}} \Big[\mathcal{R}^{+}_0 \big(V\mathcal{R}^{+}_0\big)^k - \mathcal{R}^{-}_0 \big(V\mathcal{R}^{-}_0\big)^k \Big](z)(x,y) dz \Big| \les  \la t \ra^{-{\f 32}}.
\end{align}
\end{prop}  
We use the algebraic identity
\be \label{eq:resolv diff alg}
\mathcal{R}^{+}_0 \big(V\mathcal{R}^{+}_0\big)^k - \mathcal{R}^{-}_0 \big(V\mathcal{R}^{-}_0\big)^k = \sum_{\ell=0}^{k} \big( \mathcal{R}^{-}_0V\big)^\ell  [\mathcal{R}_0^+ - \mathcal{R}_0^{-}] \big(V \mathcal{R}^{+}_0\big)^{k-\ell}. \ee
%Without loss of generality we will assume the difference falls on the first resolvent and prove
%\begin{align}\label{int:born}
% \sup_{x_0,x_k \in \R^3} \Big| \int_0^\infty e^{-it \sqrt{z^2+m^2}} \frac{z \chi(z)}{\sqrt{z^2+m^2}} \Big[ [\mathcal{R}_0^+ - \mathcal{R}_0^{-}] (V \mathcal{R}_0^+)^k \Big] (x_0,x_k)dz \Big| \les \la t\ra^{-3/2} .
% \end{align}

\begin{lemma}\label{lem:free diff}

	We have the following bounds on the first derivative of the difference of free resolvents.
	\begin{align*}
		[\mathcal{R}_0^+ - \mathcal{R}_0^{-}](z)(x,y)= \widetilde{O}_1(z).
	\end{align*}
	Furthermore,
	\begin{multline}\label{eqn:free res diff deriv}
		\partial_z [\mathcal{R}_0^+ - \mathcal{R}_0^{-}](z)(x-y)=\frac{i}{2\pi }\bigg(
		\frac{ \alpha \cdot (x-y)}{|x-y|}\bigg) \sin(z|x-y|)\\
		+\frac{z}{2\pi \sqrt{z^2+m^2}} \frac{\sin(z|x-y|)}{|x-y|}+(m\beta +\sqrt{z^2+m^2}I)\frac{\cos(z|x-y|)}{2\pi  }.
	\end{multline}

\end{lemma}

\begin{proof}

Note that
\begin{multline}\label{eq:free res diff}
[\mathcal{R}_0^+ - \mathcal{R}_0^{-}](z)(x,y) =\frac{1}{4\pi} [-i \alpha \cdot \nabla + m\beta + \sqrt{m^2+z^2} ] \bigg[\frac{e^{iz|x-y|}-e^{-iz|x-y|}}{|x-y|}\bigg] \\
   =  \frac{1}{2\pi} \alpha \cdot \nabla \bigg[ \frac{ \sin (z|x-y|)}{|x-y|} \bigg] + \frac{i}{2\pi }\big( m\beta + \sqrt{z^2 + m^2}I\big) \bigg[\frac{\sin(z|x-y|)}{|x-y|}  \bigg].
\end{multline}
Using this representation, we express the difference of free resolvents with two pieces.  We ignore the constant factors.  We first consider $ A(z,|x-y|): =   \alpha \cdot \nabla  \bigg[ \frac{ \sin (z|x-y|)}{|x-y|} \bigg] $, which satisfies the bound $ \widetilde{O}_1(z^2)$.
By direct computation, we have
\begin{align*} 
A(z,|x-y|) & =\bigg[ \alpha \cdot  \frac{(x-y)} {|x-y|}  \bigg] \bigg[\frac{z|x-y| \cos(z|x-y|) - \sin(z|x-y|)}{|x-y|^2}\bigg].
\end{align*}
%By considering the cases of $z|x-y|\gtrsim 1$ and $z|x-y|\les 1$, we see the desired bounds. 
First if
$z|x-y| \gtrsim 1$, using $|x-y|^{-1}\les z$ establishes the desired bound. To see the inequality for $z|x-y|\les 1$ note that by Taylor series expansion one has  $s\cos(s)-\sin(s)= \widetilde{O}_1(s^3)$.
Taking derivative of $A(z,|x-y|)$ we have
\begin{align*}
	\partial_z A(z,|x-y|) &=
	\bigg(\frac{ \alpha\cdot (x-y)}{|x-y|}\bigg)
	z\sin(z|x-y|).
\end{align*} 
The desired bound easily follows from this explicit representation.

We move to the second part of \eqref{eq:free res diff}
let $B(z,|x-y|):= \big( m\beta + \sqrt{z^2 + m^2}I\big) \frac{\sin(z|x-y|)}{|x-y|}$.  A direct computation shows
\begin{align*}
	\partial_z B(z,|x-y|)&= \frac{z}{\sqrt{z^2+m^2} }\frac{\sin (z|x-y|)}{|x-y|}+(m\beta +\sqrt{z^2+m^2}I)\cos (z|x-y|).
\end{align*}
As before, considering the cases of $z|x-y|\gtrsim 1$ and $z|x-y|\les 1$ separately suffices. 
\end{proof}

\begin{proof}[Proof of Proposition~\ref{lem:born}]
Using the identity \eqref{eq:resolv diff alg}, we fix $\ell$ and consider the contribution of
\be\label{int:born}
	\Big| \int_0^\infty e^{-it \sqrt{z^2+m^2}} \frac{z \chi(z)}{\sqrt{z^2+m^2}} \Big[ \big( \mathcal{R}^{-}_0V\big)^\ell  [\mathcal{R}_0^+ - \mathcal{R}_0^{-}] \big(V \mathcal{R}^{+}_0\big)^{k-\ell} \Big] (z)(x_0,x_k)dz \Big|.
\ee
For notational convenience let  $J=\{0,1,2,\dots, k \}\setminus \{\ell \}$, $J^-=\{0,1,\dots, \ell-1\}$ and $J^+=\{\ell+1,\ell+2,\dots, k \}$.
 Note that one of $J^-$ or $J^+$ may be empty.  We first establish that integral is bounded. Using the expansion \eqref{resolventex}, we have (when $0<z\ll 1$)
\begin{align}
	\mathcal{R}_0^\pm(z)(x,y)&= [ -i \alpha \cdot \nabla + m \beta + \sqrt{m^2+ z^2} I ]  \frac{e^{\pm iz|x-y|}}{4\pi |x-y|}\nn\\
	&=\left[\bigg(\frac{\alpha \cdot(x-y)}{|x-y|} \bigg)\left[\pm i z+\frac{1 }{|x-y|}
	\right]+\bigg(m\beta +\sqrt{z^2+m^2}I \bigg)\right]\frac{ e^{\pm iz|x-y|}}{4\pi |x-y|}\label{eq:resolv H}  \\
	&=e^{\pm iz|x-y|}H_1(z,x,y), \quad \sup_{|z|<z_0}|\partial_z^k H_1(z,x,y)|\les \frac{1}{|x-y|}+\frac{1}{|x-y|^2},  \nn
\end{align} 
for each $k=0,1,2,\dots.$
Furthermore,
\begin{multline}\label{eq:H2bound}
	\partial_z\mathcal{R}_0^\pm(z)(x,y) =\bigg[ iz  \frac{\alpha \cdot (x-y)}{|x-y|}   \pm  im\beta\pm i\sqrt{z^2+m^2}I+\frac{z}{|x-y|\sqrt{z^2+m^2}}  
	\bigg]\frac{e^{iz|x-y|}}{4\pi}\\
	=e^{\pm iz|x-y|}H_2(z,x,y), \qquad \sup_{|z|<z_0}|\partial_z^kH_2(z,x,y)|\les 1+ \frac{1}{|x-y|} \quad k=0,1,2\dots 
\end{multline}
From this we see, for $0<z\ll 1$,
\begin{align}\label{eq:partialzR0}
	|\partial_z^j\mathcal R_0^\pm(z)(x,y)| & \les \left( \frac{1}{|x-y|}+\frac{1}{|x-y|^2}\right) |x-y|^j, \qquad j=0,1,2.
\end{align}
Using this bound and \eqref{int:born}, the $z$ integral is clearly bounded due to the cut-off to $0<z\ll 1$,
\begin{align*}
	 \sup_{x_0,x_k \in \R^3} &\Big| \int_0^\infty e^{-it \sqrt{z^2+m^2}} \frac{z \chi(z)}{\sqrt{z^2+m^2}} \Big[ \big( \mathcal{R}^{-}_0V\big)^\ell  [\mathcal{R}_0^+ - \mathcal{R}_0^{-}] \big(V \mathcal{R}^{+}_0\big)^{k-\ell}\Big] (x_0,x_k)dz \Big| \\
	 &\les \sup_{x_0,x_k \in \R^3} 
	 \int_{\R^{3k}} \prod_{p=1}^k |V(x_p)| \prod_{j\in J} \bigg(\frac{1}{|x_j-x_{j+1}|}+\frac{1}{|x_j-x_{j+1}|^2}
	 \bigg) \, dx_1dx_2\dots dx_{k}.
\end{align*}
This is seen to be bounded uniformly in $x_0,x_k$ using Lemma~\ref{lem:potential2} to iterate the bound
$$
	\sup_{ x_{j+1} \in \R^3 } \int_{\R^3} \la x_j\ra^{-2-} \bigg(\frac{1}{|x_j-x_{j+1}|}+\frac{1}{|x_j-x_{j+1}|^2} \bigg) \, dx_j \les 1,
$$
first integrating in $x_\ell$.

To establish the time decay, we integrate by parts once then use Lemma~\ref{lem:vdc}.  Integrating by parts once leaves us to bound
\begin{align*}
	\frac{1}{t} \int_0^\infty e^{-it\sqrt{z^2+m^2}} \partial_z \left(\chi(z) \big( \mathcal{R}^{-}_0V\big)^\ell  [\mathcal{R}_0^+ - \mathcal{R}_0^{-}] \big(V \mathcal{R}^{+}_0\big)^{k-\ell} 	\right)(z)\, dz.
\end{align*}
Note that there is no boundary term since $[\mathcal{R}_0^+ - \mathcal{R}_0^{-}]=\widetilde O_1(z)$ by Lemma~\ref{lem:free diff} and by Lemma~\ref{lem:R0exp} the free resolvents are bounded with respect to $z$, and the support of $\chi$.
We consider two cases, if the derivative acts on the difference of resolvents or on a resolvent.
If the derivative acts on the cut-off function, we can easily integrate by parts again with the existing bounds.  We first consider when the derivative acts on difference of resolvents. From the representation in \eqref{eqn:free res diff deriv}, we can write
\begin{multline*}
	\partial_z [\mathcal{R}_0^+ - \mathcal{R}_0^{-}](z)(x_\ell,x_{\ell+1})=e^{iz|x_\ell-x_{\ell+1}|} A_1(z,|x_\ell-x_{\ell+1}|)\\
	+e^{-iz|x_\ell-x_{\ell+1}|}A_2(z,|x_\ell-x_{\ell+1}|)
	+ \widetilde{O}_1(z),
\end{multline*}
with 
$$
	|\partial_z^j A_1(z,|x_\ell-x_{\ell+1}|)|, \quad |\partial_z^j A_2(z,|x_\ell-x_{\ell+1}|)|\les 1, \qquad j=0,1.
$$
The error term comes because we have $\big[\frac{z}{\sqrt{z^2+m^2}}\frac{\sin(zr)}{r}\Big]=\widetilde{O}_1(z)$.
With a slight abuse of notation, we denote both the operators $A_1$ and $A_2$ by $a(z)$.
Combining this with \eqref{eq:resolv H}, we need to bound terms of the form
\begin{multline*}
	\frac{1}{t}\int_0^\infty e^{-it\sqrt{z^2+m^2}} \chi(z)
	\bigg( e^{\pm i z|x_\ell-x_{\ell+1}|} a(z)+ \widetilde{O}_1(z)  \bigg)\\
	\prod_{j\in J^-} e^{-iz|x_j-x_{j+1}|}H_1(z,x_j,x_{j+1})
	\prod_{p\in J^+} e^{iz|x_p-x_{p+1}|}H_1(z,x_p,x_{p+1})\, dz.
\end{multline*}
We apply Lemma~\ref{lem:vdc} with $$\phi(z)=-t\sqrt{z^2+m^2}-z\bigg(\sum_{j\in J^-} |x_j-x_{j+1}|+\gamma |x_\ell-x_{\ell+1}|-\sum_{p\in J^+} |x_p-x_{p+1}|\bigg),
$$ 
where $\gamma\in \{-1,0,1 \}$, and 
$$
	\psi(z)=[a(z)+ \widetilde{O}_1(z)]\prod_{j\in J} H_1(z,x_j,x_{j+1}).
$$
We may again  bound the contribution of the spatial integrals by Lemma~\ref{lem:potential2}.
\begin{align*}
	\frac{1}{t^{\f32}} \int_{\R^{3k}} \prod_{p=1}^k |V(x_p)| \prod_{j\in J} \bigg(\frac{1}{|x_j-x_{j+1}|}+\frac{1}{|x_j-x_{j+1}|^2}
		 \bigg) dx_1dx_2\dots dx_k \les \frac{1}{t^{\f32}}.
\end{align*}

On the other hand, if the derivative hits one of the iterated resolvents, we have to bound
\begin{align*}
	\frac{1}{t} \int_0^\infty e^{-it\sqrt{z^2+m^2}}  \big( \mathcal{R}^{-}_0V\big)^\ell  [\mathcal{R}_0^+ - \mathcal{R}_0^{-}] \partial_z \big(V \mathcal{R}^{+}_0\big)^{k-\ell} (z)	\, dz.
\end{align*}
Using Lemma~\ref{lem:free diff}, we have    
$[\mathcal{R}_0^+ - \mathcal{R}_0^{-}](z)(x_\ell,x_{\ell+1})= \widetilde{O}_1(z)$.  Then, using \eqref{eq:resolv H}, we have
$$
\big( \mathcal{R}^{-}_0V\big)^\ell  \partial_z \big(V \mathcal{R}^{+}_0\big)^{k-\ell}(z) = e^{iz (\sum_{p\in J^+} |x_p-x_{p+1}|-\sum_{j\in J^-} |x_j-x_{j+1}|)} b(z),
$$
where 
$$
	|b(z)|,\, |\partial_z b(z)|\les \sum_{\ell\in J^+} |x_\ell-x_{\ell+1}| \prod_{j\in J} \bigg(\frac{1}{|x_{j}-x_{j+1}|}+\frac{1}{|x_{j}-x_{j+1}|^2}\bigg)\prod_{r=1}^k V(x_r).
$$
Combining these bounds we have to bound
\begin{align*}
	\frac{1}{t} \int_0^\infty e^{-it\sqrt{z^2+m^2}+iz (\sum_{p\in J^+} |x_p-x_{p+1}|-\sum_{j\in J^-} |x_j-x_{j+1}|)} \psi(z)\, dz,
\end{align*}
where $\psi(z), \psi'(z)$ are supported on a small neighborhood of $z=0$ and satisfy
   $$
   	|\psi(z)|,\, |\partial_z \psi(z)|\les\sum_{\ell\in J^+} |x_\ell-x_{\ell+1}| \prod_{j\in J} \bigg(\frac{1}{|x_{j}-x_{j+1}|}+\frac{1}{|x_{j}-x_{j+1}|^2}\bigg)\prod_{r=1}^k V(x_r).
   $$
Thus, we apply Lemma~\ref{lem:vdc} to bound the spatial integral
\begin{align*}
	\sup_{x_0,x_k\in \R^3} \frac{1}{t^{\f32}} \int_{\R^{3k}} 
	\sum_{\ell\in J^+} |x_\ell-x_{\ell+1}| \prod_{j\in J} \bigg(\frac{1}{|x_{j}-x_{j+1}|}+\frac{1}{|x_{j}-x_{j+1}|^2}\bigg)\prod_{r=1}^k \la x_r\ra^{-3-}\, dx_1dx_2\dots dx_k.
\end{align*}
Using Lemma~\ref{lem:potential2}, first in $x_{\ell}$, we show that the spatial integrals are bounded uniformly in $x_0, x_{k+1}$ by iterating the bound
$$
	\sup_{x_{j+1}\in \R^3}\int_{\R^3} \la x_j \ra^{-3-}\bigg(
	1+\frac{1}{|x_j-x_{j+1}|}+\frac{1}{|x_j-x_{j+1}|^2}
	\bigg)\, dx_j \les 1.
$$

\end{proof}

We finish this subsection with the following general lemma which will be useful in the following subsections.

\begin{lemma}\label{lem:genericbound}
Assume that the operator $E(z)$ with kernel $E(z)(x,y)$ satisfies (for $0<|z|<z_0$)
$$\|  |\partial_z^{k} E(z)(x,y)| \|_{L^2 \rightarrow L^2} \les 1,\,\,\,   k=0,1,\,\,\text{ and }
 \| | \partial_z^{2} E(z)(x,y)| \|_{L^2 \rightarrow L^2} \les z^{-1+}.$$  Also assume that the operators $E_1(z)$ and $E_2(z)$ 
satisfy (for some $\alpha\geq 0$)
$$
\big|\partial^k_z E_j(z)(x,y)\big|\les  \big(|x-y|^{-2}+|x-y|^\alpha\big), \,\,\,\,j=0,1, \,\,\,k=0,1,
\,\,\text{ and }$$
$$
\big|\partial^2_z E_j(z)(x,y)\big|\les z^{-1+} \big(|x-y|^{-2}+|x-y|^\alpha\big), \,\,\,\,j=0,1.
$$
Let $|V(x)| \les \la x \ra ^{-\beta}$ for some $\beta>2\alpha+3$.  Then, 
$$
\sup_{x,y\in \R^3} \Bigg| \int_{-\infty}^{\infty} e^{-it \sqrt{z^2+m^2}} \frac{z \chi(z)}{\sqrt{m^2+z^2}} \Big( \mathcal{R}_0VE_1v^*EvE_2V\mathcal{R}_0  
  \Big)(z)(x,y)   dz  \Bigg| \les \la t\ra^{-\f32}.
$$
\end{lemma}
\begin{proof}
We start with bound for small $t$. Using the bounds  in the hypothesis for $k=0$ and using $|\mR_0(z)(x,y)|\les 1+|x-y|^{-2} $ from \eqref{eq:resolv H}, we estimate the $z$ integral by 
$$
\int_{-\infty}^{\infty}   \chi(z)\psi(z)dz,\,\,\,\text{ where}
$$
$$
\psi(z)=\int_{\R^{12}}   \frac{(1+r_1^{-2}) (r_2^{-2}+r_2^\alpha)}{ \la x_1\ra^{ \beta}\la x_2\ra^{ \f\beta2}}  |E(z)(x_2,y_2)| 
\frac { (r_3^{-2}+r_3^\alpha) (1+r_4^{-2})}{\la y_2\ra^{ \f\beta2}\la y_1\ra^{ \beta}} dx_1dx_2dy_1dy_2.
$$
Here $r_1:=|x-x_1|, r_2:=|x_1-x_2|,r_3:=|y_2-y_1|, r_4:=|y_1-y|$.
We can bound $\psi$  by 
$$
\Big\|\int_{\R^3} \frac{(1+r_1^{-2}) (r_2^{-2}+r_2^\alpha)}{ \la x_1\ra^{ \beta}\la x_2\ra^{ \f\beta2}} dx_1\Big\|_{L^2_{x_2}(\R^3)}^2 
\big\||E(z)|\big\|_{L^2\to L^2}
\Big\|\int_{\R^3} \frac{(1+r_4^{-2}) (r_3^{-2}+r_3^\alpha)}{ \la y_1\ra^{ \beta}\la y_2\ra^{ \f\beta2}} dy_1\Big\|_{L^2_{y_2}(\R^3)}^2 .
$$
Note that  using Lemma~\ref{lem:potential2}
\begin{multline*}
\int_{\R^3}\frac{(1+r_1^{-2}) (r_2^{-2}+r_2^\alpha)}{ \la x_1\ra^{ \beta}\la x_2\ra^{ \f\beta2}} dx_1 \les \int_{\R^3}
\frac{1 }{ \la x_1\ra^{ \beta-\alpha}\la x_2\ra^{ \f\beta2-\alpha} } dx_1 + \int_{\R^3} r_1^{-2} r_2^{-2} \la x_2\ra^{-\f\beta2} dx_1 \\
\les 
\la x_2\ra^{ -\f\beta2+\alpha } + |x-x_2|^{-1}  \la x_2\ra^{-\f\beta2} \in L^2_{x_2},
\end{multline*}
uniformly in $x$ provided that $\beta>2\alpha+3$. This finishes the proof since $\big\||E(z)|\big\|_{L^2\to L^2}$ is bounded on the support of $\chi$. 

Now we consider the claim for large $t$. After an integration by parts we have to bound
$$
  \frac1{t} \int_{-\infty}^{\infty}   e^{it \sqrt{z^2+m^2}} \partial_z   \Big( \chi(z) [\mathcal{R}_0VE_1v^*EvE_2V\mathcal{R}_0  
 ](x,y)  \Big)   dz.  
$$
Now  using Lemma~\ref{lem:vdc} with the phase $\phi=t\sqrt{z^2+m^2}+zr_1+zr_4$ we estimate the integral above by
$$
  \frac1{|t|^{\f32}} \int_{-\infty}^{\infty}    \Big| \partial_z \Big[e^{-iz(r_1+r_4)} \partial_z   \Big( \chi(z) [\mathcal{R}_0VE_1v^*EvE_2V\mathcal{R}_0  
 ](x,y)  \Big) \Big]  \Big|dz. 
$$
Note that using \eqref{eq:resolv H} and \eqref{eq:H2bound} we have
$$
|\mR_0|, |\partial_z\mR_0|, \big|\partial_z  e^{-iz|x-y|} \mR_0\big|,  \big|\partial_z  e^{-iz|x-y|} \partial_z\mR_0\big| \les 1+|x-y|^{-2}.
$$
The proof now follows from the calculation above for small $t$; the only difference is, if both derivatives hit $E$ (or one of $E_1$, $E_2$), the $z$ integral will have a harmless $z^{-1+}$ term, which is integrable on the support of $\chi(z)$.  
\end{proof}

\subsection{Dispersive estimates when there is a resonance of the first kind} \label{subsec:noe}
  
  In this subsection we consider the case when there is a resonance of the first kind at threshold energy, that is when $S_1\neq 0$ and $S_2=0$, in which case $S_1$ is rank at most two by Corollary~\ref{cor:2d}. 
  
In the previous section we established the contribution of the first three terms in the expansion \eqref{eq:pert2} to the Stone's formula \eqref{eq:Stone}. Now we turn to the last term in \eqref{eq:pert2}, we need to analyze 
\begin{multline} \label{eq:beforeextend}
 \int_{0}^{\infty} e^{-it \sqrt{z^2+m^2}} \frac{z \chi(z)}{\sqrt{z^2+m^2}} \Big[[\mathcal{R}_0^{+}V\mathcal{R}_0^{+}v^*(A^{+})^{-1}  v \mathcal{R}_0^{+}V\mathcal{R}_0^{+}](z) \\
  - [\mathcal{R}_0^{-}V\mathcal{R}_0^{-}v^*(A^{-})^{-1}  v \mathcal{R}_0^{-}V\mathcal{R}_0^{-}](z) \Big] dz.
\end{multline}
Recalling the discussion immediately preceeding Lemma~\ref{lem:M0M1}, we identify
%Using  \eqref{free} and \eqref{resolventex}; the definitions of $R_0(z)$ and $\mathcal{R}_0(z)$, one can easily conclude that 
$ \mathcal{R}_0^{-}(-z)=\mathcal{R}_0^{+}(z)=:\mathcal{R}_0(z)$.  Similarly, $A^-(-z)=A^+(z):=A(z)$. Hence, by a change of variable we can extend the integral \eqref{eq:beforeextend} to the whole real line and obtain 
$$
\eqref{eq:beforeextend}= \int_{-\infty}^{\infty} e^{-it \sqrt{z^2+m^2}} \frac{z \chi(z)}{\sqrt{z^2+m^2}} [\mathcal{R}_0V\mathcal{R}_0v^*A^{-1}  v \mathcal{R}_0V\mathcal{R}_0](z)(x,y) dz.
$$
In contrast to the analysis of the Born series in the previous subsection, we extend the integral to the real line.  This will allow us to integrate by parts without boundary terms and, after a change of variables, use Fourier transform techniques.

Note that we have 
\begin{multline} \label{eq:bound1}
\sup_{x,y\in \R^3} \left|\int_{-\infty}^{\infty} e^{-it \sqrt{z^2+m^2}} \frac{z \chi(z)}{\sqrt{z^2+m^2}} [\mathcal{R}_0V\mathcal{R}_0v^* A^{-1}  v \mathcal{R}_0V\mathcal{R}_0] (z)(x,y) dz\right|\\
  \les \sup_{x,y,|z|\leq z_0}  \big| [\mathcal{R}_0(z)V\mathcal{R}_0(z)v^*[z A^{-1} (z)] v \mathcal{R}_0(z)V\mathcal{R}_0(z)](x,y)\big|.
\end{multline}
By Lemma~\ref{lem:invexp1},  $|z|\, \| A^{-1}(z)(x,y) \|_{L^2 \rightarrow L^2} \les 1$ on the support of $\chi$. Then, by Remark~\ref{rmk:potential} we have 
$$
|\eqref{eq:bound1}|   
  \les \sup_{x,y\in \R^3} \| [\mathcal{R}_0V \mathcal{R}_0v^*](x,x_2))\|_{L^{2}_{x_2}} \||z  A^{-1} (z) |\|_{L^2 \rightarrow L^2} \|  v \mathcal{R}_0V \mathcal{R}_0](y_2,y)\| _{L^{2}_{y_2}} \les 1,
$$
 which shows the boundedness of \eqref{eq:bound1} as $t\rightarrow 0$. Hence, to establish the claim of Theorem~\ref{thm:main1}, it will be enough to prove the following proposition for any $t>1$.

\begin{prop}\label{prop:res1} 
	Under the assumptions of Theorem~\ref{thm:main1}, we have
	\begin{align*}
		\int_{-\infty}^{\infty} e^{-it\sqrt{z^2+m^2}} \frac{z\chi(z)}{\sqrt{z^2+m^2} }
		[\mathcal{R}_0 V\mathcal{R}_0 v^*A^{-1} v \mathcal{R}_0 V\mathcal{R}_0 ](z)(x,y) \, dz 
		=t^{-\f12}e^{-imt} K_t(x,y)+O(t^{-\f32}),
	\end{align*}
where the error term holds uniformly in $x,y$; $K_t(x,y)=P_r(x,y)+\widetilde K_t(x,y)$  is a time dependent operator of rank at most $ 2$  satisfying $\sup_t\|K_t\|_{L^1\to L^\infty}\les 1$ and $|\widetilde K_t(x,y) |\les \la x\ra^{j} \la y\ra^{j}\la t\ra^{-j}$ for any $0\leq j\leq 1$. Moreover,  
\begin{multline*}
P_r(x,y) = \sum_{j=1}^2 c_j   \psi_j(x) \psi_j^*(y), \text{ where }   c_j= \frac{(-2\pi i)^{\f32}} {m^{\f32}\|M_{uc}V\psi_j \|^2_{\C^4}} \text{ and }\\
 \psi_j \in  L^{2,-\f12-}(\R^3)\cap L^\infty(\R^3), \,\ ( D_m +V -mI)\psi_j=0, \,\, \\ 
\la  M_{uc}V\psi_i ,  M_{uc}V\psi_i \ra = \|M_{uc}V\psi_j \|^2_{\C^4} \delta_{ij},  \,\,i,j=1,2.
 \end{multline*}
Here $c_2=0$ iff $\text{rank} (S_1)=1$. 
\end{prop}
  To establish Proposition~\ref{prop:res1}, using the expansion in Lemma~\ref{lem:invexp1}, it suffices   to consider the following integrals
\be  \label{int:pert}\begin{split}
 \int_{-\infty}^{\infty} e^{-it \sqrt{z^2+m^2}} \frac{ -i \chi(z)}{\sqrt{z^2+m^2}}\big[\mathcal{R}_0(z)V\mathcal{R}_0(z)v^*S_1D_1S_1v\mathcal{R}_0(z)V\mathcal{R}_0(z) \big] (x,y) dz,  \\ 
   \int_{-\infty}^{\infty} e^{-it \sqrt{z^2+m^2}} \frac{z \chi(z)}{\sqrt{z^2+m^2}}\big[\mathcal{R}_0(z)V\mathcal{R}_0(z)v^*E(z)v\mathcal{R}_0(z)V\mathcal{R}_0(z) \big] (x,y) dz.
\end{split}
\ee
The second integral is $O(\la t\ra^{-3/2})$ using Lemma~\ref{lem:genericbound} provided that $\beta>5$. Indeed, the required bound for $E$ is given in  Lemma~\ref{lem:invexp1}, and for $E_1=E_2=\mR_0$ the hypothesis is satisfied with $\alpha=1$ using \eqref{eq:partialzR0}.  

Now we consider the first integral in \eqref{int:pert}. Using   \eqref{resolventex} for $\mathcal{R}_0(z)$ and  \eqref{free},   
and letting $F(x,y):= \frac{1}{4\pi}[   i \alpha \cdot \frac{(x-y)}{|x-y|^2}+ 2m I_{uc}] $, we have
\be \label{eq:HF} 
\mathcal{R}_0(z)(x,y)    = F(x,y)\frac{e^{i z|x-y|}}{|x-y|} +  \bigg[  i z \alpha \cdot \frac{(x-y)}{|x-y|} +  (\sqrt{z^2+m^2} -m)I\bigg]  \frac{e^{i z|x-y|}}{4\pi |x-y|}.                                 
\ee
Hence,
\begin{multline} \label{eq:Sint}
\mathcal{R}_0(z)(x,x_1) \mathcal{R}_0(z)(x_1,x_2) \mathcal{R}_0(z)(y_2,y_1)\mathcal{R}_0(z)(y_1,y) \\ = F(x,x_1)F(x_1,x_2)F(y_2,y_1)F(y_1,y)  \frac{e^{iz \theta }}{r_1r_2r_3r_4} + z \mathcal{E}(z)  e^{iz \theta },
\end{multline}
where $\theta = |x-x_1|+|x_1-x_2|+|y_2- y_1|+|y_1-y|:= r_1+ r_2 + r_3 +r_4 $ and $\mathcal{E}(z) $ satisfies the bound
$$
|\mathcal{E}^{(j)}(z)| \les \prod_{i=1}^{4}\bigg({\f 1 {r_i^2}}+{\f 1 {r_i}} \bigg), \,\,\,\,\, j=0,1,2.
$$
Therefore, for the first term in \eqref{int:pert} is given by
\begin{multline} \label{eq:finite+E}
 \frac{F(x,x_1)F(x_1,x_2)F(y_2,y_1)F(y_1,y)}{r_1r_2r_3r_4}\int_{-\infty}^{\infty} e^{-it \sqrt{z^2+m^2}+iz\theta} \frac{\chi(z)}{\sqrt{z^2+m^2}} dz \\
   + \int_{-\infty}^{\infty} e^{-it \sqrt{z^2+m^2}} \frac{z \chi(z)}{\sqrt{z^2+m^2}} \mathcal{E}(z) e^{iz \theta} dz = I + II.
 \end{multline} 
Note that $II$ can be estimated as follows using integration by parts followed with Lemma~\ref{lem:vdc},
\begin{multline*}
|II|= \left|{\f C t} \int_{-\infty}^{\infty} e^{-it \sqrt{z^2+m^2}+{iz \theta}} \Big[ \big(\chi(z) \mathcal{E}(z)\big)^\prime + \theta \chi(z) \mathcal{E}(z) \Big]dz \right|\\
   \les {\f 1 {|t|^{3/2}} }  \Big| \int_{-1}^{1}  \big(\chi(z) \mathcal{E}(z)\big)^{\prime\prime} + \theta \big(\chi(z) \mathcal{E}(z) \big)^\prime dz\Big|  \les {\f 1{ |t|^{3/2} }} \la \max_i{r_i}  \ra  \prod_{i=1}^{4}({\f 1 {r_i^2}}+{\f 1 r_i}).
\end{multline*} 
The spatial integrals can be estimated as in the proof of Lemma~\ref{lem:genericbound} with $\alpha=0$, $\beta>3$, and $E=S_1D_1S_1$. 

 Next we consider the first term in \eqref{eq:finite+E}. Note that this integral can be estimated by $t^{-\f12}$ easily using Lemma~\ref{lem:vdc} with $\phi(z)=-t \sqrt{z^2+m^2}+{z \theta}$. In the rest of this subsection we establish the properties of the operator which has decay rate $t^{-\f12}$.

For notational convenience, we suppress the integral kernels' spatial variable dependence, which should be clear from context. First we assume that at least one of the $r_j$'s is greater than $t$. In this case we have $ \frac{1}{ \max_j r_j} \les {\f 1 t} $. Hence, we can exchange the largest $r_j$ with $t$ to gain extra time decay. Using an analysis similar to that in the proof of Lemma~\ref{lem:genericbound} one can easily see that the spatial integrals converge.   Thus, we have
\begin{align*}
 |I | \les {\f1{t^{\f12}}} \Big| \int_{\R^{12}} \frac{FVFv^*S_1D_1S_1vFVF}{r_1 r_2 r_3 r_4} dx_1dx_2dy_1dy_2 \Big| \les t^{-\f32}.
\end{align*}
 Now it remains to consider the case when $r_j \ll t$ for all $j$. We start with  the following lemma.  
 %\begin{align} \label{eq:allsmall}
 %\prod_{j=1}^4 \chi\big(\frac{r_j}{t}\big) \int_{-\infty}^{\infty} e^{-it \sqrt{z^2+m^2}+iz\theta} \frac{\chi(z)}{\sqrt{z^2+m^2}} dz.
 % \end{align} 
% We give Lemma~\ref{lem:finite} to show that the slow decaying part is a rank one operator. 
\begin{lemma} \label{lem:finite} Let $\theta = \sum_{j=1}^{4} r_j$. Then, 
\begin{multline}\label{eq:allsmall}
 \prod_{j=1}^4 \chi\big(\frac{r_j}{t}\big) \int_{-\infty}^{\infty} e^{-it \sqrt{z^2+m^2}+iz\theta} \frac{\chi(z)}{\sqrt{z^2+m^2}} dz = \frac{(-2\pi i)^{\f12} e^{3imt}}{(mt)^{\f12}} \prod_{j=1}^4 \chi\big(\frac{r_j}{t}\big) e^{-im(t^2-r_j^2)^{\f12}} \\
+ O\Bigg( \frac{1}{t^{\f32}}\Big[ \sum_{1\leq i<j\leq 4} r_ir_j+\sum_{j=1}^4 r_j +1 \Big] \Bigg) .
\end{multline}
\end{lemma}
For the proof of Lemma~\ref{lem:finite} we need the following lemma.
\begin{lemma}\label{lem:sum}  
Let $f: \R \rightarrow \R$ be $C^2$ with bounded derivatives. Then for any $a_i > 0$ we have
$$
f\Big( \sum_{j=1}^n a_j \Big)=\sum_{j=1}^n f(a_j) - (n-1)f(0) + O \Big( \sum_{1\leq i< j\leq n}  a_ia_j \Big).
$$
\end{lemma}
\begin{proof} By a simple induction argument, it suffices to prove this for $n=2$.  
 Without loss of generality we can also  assume that $a_2\geq a_1$. By the mean value theorem, we have
\begin{align*}
f(a_1+a_2) &= f(a_1)+ f(a_2) + [f(a_1+a_2) - f(a_2)] - [f(a_1)-f(0)] - f(0) \\
& = f(a_1)+ f(a_2) + f^\prime (c_1) a_1- f^\prime (c_2) a_1 -f(0) \,\,\,\ \text{for} \,\ c_1 \in(a_2, a_2+a_1) \ , \ c_2 \in (0,a_1) \\
& = f(a_1)+ f(a_2) + a_1(c_1-c_2) f^{\prime \prime} (c)- f(0).  
\end{align*}
Since $0\leq a_2-a_1\leq c_1-c_2\leq a_1+a_2\leq 2a_2$, this yields the lemma.
\end{proof}

\begin{proof}[Proof of Lemma~\ref{lem:finite}]
For the sake of simplicity we prove the lemma for $m=1$. Note that first, the critical point of $\phi(z)=  \sqrt{z^2+1}-{ z \gamma}$ is $ \omega = \frac{\gamma}{\sqrt{1-\gamma^2}}$. Here $\omega$ is defined since  $r_j \ll t$ for all $j$ implies $\gamma=\frac{\theta}{t} \ll 1$. We use the change of variables $z\mapsto z+\omega$ to move the critical point to zero and write
\begin{align*}
  \eqref{eq:allsmall}=\prod_{j=1}^4 \chi\big(\frac{r_j}{t}\big) e^{-it \sqrt{1-\gamma^2}} \int_{-\infty}^{\infty} e^{-it\big(\sqrt{(z+\omega)^2+1}-z\gamma -{\frac{1}{\sqrt{1-\gamma^2}}}\big)} \frac{\chi(z+\omega)}{\sqrt{(z+\omega)^2+1}} dz.
\end{align*}
With a change of variable $z =g(s) = \frac{s}{ \sqrt{1 - \gamma^2}} \big( \sqrt{1+ \frac{s^2}{4} } + \frac{s \gamma}{2} \big) $ this integral can be written as
$$ 
  \eqref{eq:allsmall}= \prod_{j=1}^4 \chi\big(\frac{r_j}{t}\big) e^{-it \sqrt{1-\gamma^2}} \int_{-\infty}^{\infty} e^{-it s^2 \frac{\sqrt{1-\gamma^2}}{2}} \psi(s) ds, $$
where
$$
\psi(s):= \frac{\chi(g(s)+\omega)}{  \sqrt{(g(s)+\omega)^2+1}}  g^\prime(s).
$$
Note that $\psi$ is supported on $\{s: |s|\les 1\}$. Since on this set   $|g^{(k)}(s)| \les 1$ for all $k \geq0$, we see that $\psi$ is a Schwartz function with derivatives bounded uniformly in $\gamma\ll 1$. 
Then, we have
\begin{align*} 
 \int_{-\infty}^{\infty} e^{-it s^2 \frac{\sqrt{1-\gamma^2}}{2}} \psi(s) ds  & = (-2\pi i)^{\f12}
 \int_{-\infty}^{\infty} \F^{-1}  \Big(\frac{e^{-it \frac{\sqrt{1-\gamma^2}}{2} (\cdot)^2}}{(1-\gamma^2)^{\f14} t^{\f12}} \Big)(\xi)  \widehat{\psi}(\xi) d \xi 
 \\& = \frac{(-2\pi i)^{\f12}}{(1-\gamma^2)^{\f14} t^{\f12}} \int_{-\infty}^{\infty} e^{-i{\frac{2 \xi^2}{t\sqrt{1-\gamma^2}}}} \widehat{\psi}(\xi) d \xi 
 \\ &= \frac{(-2\pi i)^{\f12}}{(1-\gamma^2)^{\f14} t^{\f12}}\Bigg[\psi(0) + \int_{-\infty}^{\infty} \Big[e^{-i{\frac{ 2 \xi^2}{t\sqrt{1-\gamma^2}}}} -1\Big] \widehat{\psi}(\xi) d \xi \Bigg].                                                                 
\end{align*}
Note that for $ \gamma \ll1$ we have $(1-\gamma^2)^{-\f14} \les 1$, and 
\begin{align*}
 \frac{1}{(1-\gamma^2)^{\f14} t^{\f12}} \left|\int_{-\infty}^{\infty}  \Big[e^{-i{\frac{2 \xi^2}{t\sqrt{1-\gamma^2}}}} -1\Big] \widehat{\psi(\xi)} d \xi \right|& \les  \frac1{t^{\f32}} \int_{-\infty}^{\infty} |\xi^2 \widehat{\psi}(\xi) | d\xi \\ 
		       & \les  \frac1{t^{\f32}} \|\psi\|_{L^1} +  \frac1{t^{\f32}} \|\psi^{\prime\prime\prime\prime}\|_{L^1} \les  \frac1{t^{\f32}}.
\end{align*}  
Hence, this term has the contribution $O(t^{-\f32})$ to \eqref{eq:allsmall}. For the last equality we used the fact that $ \| \partial_z^k \psi\|_{L^1} \les 1 $ uniformly in $\gamma$.

 We are left with the contribution of $\psi(0)$ to  \eqref{eq:allsmall} which is given by
\begin{align*}
 \prod_{j=1}^4 \chi\big(\frac{r_j}{t}\big) t^{-\f12} e^{-it \sqrt{1-\gamma^2}}   \frac{\chi( \omega) }{\sqrt{ \omega^2+1}(1-\gamma^2)^{\f34}} = t^{-\f12}  \prod_{j=1}^4 \chi\big(\frac{r_j}{t}\big) e^{-it f(\gamma)}   \frac{ \chi( \frac{\gamma}{\sqrt{1-\gamma^2}}) }{ (1-\gamma^2)^{\f14}},
\end{align*}
where $f(\gamma)=\sqrt{1-\gamma^2}\chi(\gamma/4)$ with $\gamma=\frac{\theta}{t}$. Note that $f$ has bounded derivatives. Since $f(0)=1$, using Lemma~\ref{lem:sum} we obtain (in the support of $ \prod_{j=1}^4 \chi\big(\frac{r_j}{t}\big)$)
\begin{multline} \label{eq:gamma}
e^{-it f(\gamma)}= e^{-it \big( \sum_{j=1}^4 f(\frac{r_j}{t}) - 3\big)} + e^{-i t \big( \sum_{j=1}^4 f(\frac{r_j}{t}) - 3 \big)} O \Big( e^{i\frac1t \sum_{1\leq i< j\leq 4} r_ir_j }-1\Big) \\
=e^{i3t}\prod_{j=1}^4 e^{-i(t^2-r_j^2)^{\f12}} + O\Big(\frac1t \sum_{1\leq i< j\leq 4}  r_ir_j\Big).
\end{multline}
Further, since $\gamma\ll 1$, we have
$$
\frac{ \chi( \frac{\gamma}{\sqrt{1-\gamma^2}}) }{ (1-\gamma^2)^{\f14}}=1+O(\gamma)=1+ O\Big(\frac1t\sum_{j=1}^k r_j\Big),
$$
we see the contribution of $\psi(0)$ to \eqref{eq:allsmall} is 
$$\frac{(-2\pi i)^{\f12} e^{i3t}}{t^{\f12}} \prod_{i=1}^4 e^{-i(t^2-r_j^2)^{\f12}} \chi (r_j/t) \\
+ O\Bigg( \frac{1}{t^{\f32}}\Big[ \sum_{1\leq i<j\leq 4} r_ir_j+\sum_{j=1}^4 r_j\Big] \Bigg).$$
This finishes the proof.
\end{proof}

We can now prove the main claim of this subsection.

\begin{proof}[Proof of Proposition~\ref{prop:res1}]

Using Lemma~\ref{lem:finite} we see that the contribution of $I$ in \eqref{eq:finite+E} to the first integral in \eqref{int:pert} is given by
\begin{multline}
-\frac{i(-2\pi i)^{\f12}}{\sqrt{m}t^{\f12}} \int_{\R^{12}}  e^{3imt}\prod_{j=1}^4 \chi\big(r_j/t \big) e^{-im(t^2-r_j^2)^{\f12}}\frac{FVFv^*S_1D_1S_1vFVF}{r_1r_2r_3r_4} dy_1dy_2dx_1dx_2 \\
  + O\Bigg(\frac{1}{t^{\f32}} \int_{\R^{12}}   \Big[ \sum_{1\leq i<j\leq 4} r_ir_j+\sum_{j=1}^4 r_j +1 \Big] \Big|\frac{FVFv^*S_1D_1S_1vFVF}{r_1r_2r_3r_4}\Big| dy_1dy_2dx_1dx_2  \Bigg)\\
  =:  t^{-\f12} K_t(x,y) + O( t^{-\f32} ).
\end{multline}
The last inequality follows from the proof of Lemma~\ref{lem:genericbound} noting that
$$ \frac{  \sum_{1\leq i<j\leq 4} r_ir_j+\sum_{j=1}^4 r_j +1  }{r_1r_2r_3r_4}\les \prod_{j=1}^4 (1+r_j^{-1}).
$$ 
Note that 
$$
K_t=  Ce^{3imt}\widetilde F_t V\widetilde F_t v^* S_1D_1S_1 v \widetilde F_t V\widetilde F_t, 
$$
where $C= (-i)^{\f32}(2 \pi )^{\f12}m^{-\f12}$ and $\widetilde{F}_t$ is an integral operator with kernel
$$
\widetilde F_t(x,y)=\frac{\chi\big( |x-y|/t\big)e^{-im[t^2-|x-y|^2]^{\f12}} F(x,y)}{|x-y|}= \chi\big( |x-y|/t\big)  e^{-imt[1-(|x-y|/t)^2]^{\f12}} \mathcal{G}_0(x,y).
$$
In particular, since $S_1$ is of rank at most two, $K_t$ is of rank at most two.

Note that since $\frac{|x-y|}{t} \les 1 $, we have 
\begin{multline*}
\chi\big( |x-y|/t\big)e^{-im[t^2-|x-y|^2]^{\f12}}  = e^{-imt} + O\big(   |x-y|^2/t \big) \\ =e^{-imt} + O\big( |x-y| \frac{|x-y|^j}{t^j} \big),\,\,\,0\leq j\leq 1.
\end{multline*}
The last equality holds since $\frac{|x-y|}{t}\les 1$.
Using this we write
$$
 K_t(x,y) =  C e^{-imt} [\mathcal{G}_0V\mathcal{G}_0v^*S_1D_1S_1v\mathcal{G}_0V\mathcal{G}_0] (x,y) +   \widetilde{K}_t(x,y).
$$
Since $ \big|[|x-y|^{1+j}\mathcal{G}_0(x,y)]\big| \les \la x\ra^j \la y \ra^j (1+|x-y|^{-1})$, we employ a similar argument as in Lemma~\ref{lem:genericbound} to show that $ |\widetilde{K}_t(x,y) |  \les \la x\ra^{j}\la y\ra^{j} t^{ -j}$, $0\leq j\leq 1$.   

By Corollary~\ref{cor:2d}, we know that the rank of $S_1$ is at most two. Hence, we can write 
$$ S_1=  \phi_1(x) \phi_1^*(y) + \phi_2(x) \phi_2^*(y) $$
where we pick $\{\phi_1,\phi_2\}$ as the orthonormal basis of $S_1L^2$. The self-adjointness of  $S_1v\mathcal{G}_1v^*S_1$  also allows us to pick the basis so that $S_1v\mathcal{G}_1v^*S_1$ is diagonal in $S_1L^2$, i.e., 
  \begin{align} \label{eq:M_uc}
   \la M_{uc} v^* \phi_j , M_{uc} v^* \phi_i \ra = \| M_{uc}v^*\phi_j \|^2_{\C^4} \delta_{ij}, \,\, i,j=1,2 .
  \end{align}
 Using this one can show that 
$$ S_1v\mathcal{G}_1v^* S_1 = \frac{m} {2\pi}\left[
  \begin{array}{ c c }
     a^2 & 0 \\
     0 & b^2
  \end{array} \right]S_1, \,\ \text{and} \,\ D_1= (S_1v\mathcal{G}_1v^* S_1)^{-1}=\frac{2\pi}{m}\left[
  \begin{array}{ c c }
     \f 1{a^2} & 0 \\
     0 & \f 1{b^2}
  \end{array} \right] S_1, $$ 
where $a= \| M_{uc}v^*\phi_1 \|_{\C^4}$ and $b= \| M_{uc}v^*\phi_2 \|_{\C^4}$. 
Therefore, the self-adjoint operator $S_1D_1S_1$ can be rewritten as 
\begin{align*}
	[S_1D_1S_1](x,y)= \frac{2\pi}{ma^2}\phi_1(x) \phi_1^*(y)+\frac{2\pi}{mb^2} \phi_2(x) \phi_2^*(y).
\end{align*}
Furthermore, Lemma~\ref{lem:esa1} gives us that $ \phi_j= Uv\psi_j$ for $\psi_j=-\mathcal{G}_0v^*\phi_j$ where $ ( D_m +V -mI)\psi_j=0$. Using this $\eqref{eq:M_uc} =\la M_{uc}V\psi_i,M_{uc}V\psi_j\ra$. 
Noting that by definition of $S_1$, we have $-S_1= S_1v\mathcal{G}_0v^*U=Uv\mathcal{G}_0v^*S_1$, we obtain
\begin{multline*}
 [\mathcal{G}_0V\mathcal{G}_0v^*S_1D_1S_1v\mathcal{G}_0V\mathcal{G}_0] (x,y) = [\mathcal{G}_0v^*S_1D_1S_1v\mathcal{G}_0](x,y) \\ = 
  \frac{2\pi }{ma^2}   [\mathcal{G}_0v^*\phi_1](x)[\mathcal{G}_0v^*\phi_1]^*(y) +  \frac{2\pi }{mb^2}  [\mathcal{G}_0v^*\phi_2](x)[\mathcal{G}_0v^*\phi_2]^*(y) \\
  =  \frac{2\pi }{ma^2}   \psi_1(x) \psi_1^*(y) +\frac{2\pi }{mb^2}   \psi_2(x) \psi_2^*(y):= \frac{m^{\f12}}{(-i)^{\f32}(2 \pi )^{\f12}}P_r(x,y).
\end{multline*}  
Finally, note that if $S_1$ is one dimensional it is generated by a single $\phi(x) $ with  $\la \phi, \phi \ra =1$. In this case we obtain $P_r(x,y) =  \frac{(-2\pi i)^{\f32}}{m^{\f32}  \| M_{uc}v^*\phi \|^2_{\C^4}}   \psi (x) \psi^*(y) $. 

This finishes the proof of Proposition~\ref{prop:res1}. 
\end{proof}

We are now ready to prove Theorem~\ref{thm:main1}.

\begin{proof}[Proof of Theorem~\ref{thm:main1}]
	
	Using the Stone's formula, \eqref{eq:Stone}, and the expansion for the resolvent \eqref{eq:pert2}, we reduce our analysis to oscillatory integral bounds.  Proposition~\ref{lem:born} suffices to bound the contribution of the first three terms of \eqref{eq:pert2} by $\la t\ra^{-\f32}$ as an operator from $L^1$ to $L^\infty$.  The contribution of the final term in \eqref{eq:pert2} is controlled by Proposition~\ref{prop:res1}.
\end{proof}

\subsection{Dispersive estimate when there is a resonance of the second or third kind at the threshold.} \label{subsec:e}

In this section we will investigate dispersive estimate in the case when $S_2 \neq 0$.
To establish the claim of Theorem~\ref{thm:main2}, we devote this subsection to proving

\begin{prop}\label{prop:res2}

	Under the assumptions of Theorem~\ref{thm:main2}, there is a finite rank operator $K_t$ so that
	\begin{align}\label{eq:resint}
		\int_{-\infty}^{\infty} e^{-it\sqrt{z^2+m^2}} \frac{z\chi(z)}{\sqrt{z^2+m^2} }
		[\mathcal{R}_0 V\mathcal{R}_0 v^*A^{-1} v \mathcal{R}_0 V\mathcal{R}_0 ](z)(x,y) \, dz
		=t^{-\f12}K_t(x,y)+ O(t^{-\f32}),
	\end{align}
 where $\sup_t \|K_t(x,y)\|_{L^1 \rightarrow L^{\infty} } < \infty$ and the error term is bounded uniformly in $x,y$. Moreover, the integral above is bounded uniformly  in $x,y,t$.
\end{prop}
\begin{proof}
Using the expansion for $A^{-1}(z)$ from Lemma~\ref{lem:generalainverse} in the integral above  we consider
 \begin{align*} 
\frac{1}{z^2} \mathcal{R}_0V \mathcal{R}_0v^*S_2D_3S_2 v \mathcal{R}_0V \mathcal{R}_0 + \frac{1}{z}\mathcal{R}_0V \mathcal{R}_0v^*\Omega v \mathcal{R}_0V \mathcal{R}_0 + \mathcal{R}_0V \mathcal{R}_0v^*E v \mathcal{R}_0V \mathcal{R}_0.
 \end{align*}
The last two terms can be handled similar to those we have already bounded in Subsection~\ref{subsec:noe}. In particular, the operator with decay rate $t^{-\f12}$ is not necessarily rank at most two, but is finite rank. This is because instead of $S_1D_1S_1$ here we have $\Omega$, which was shown to be finite rank in the proof of Lemma~\ref{lem:generalainverse}. Hence, it suffices to consider the integral
\begin{align}\label{eq:genericmainint}
  \int_{-\infty}^{\infty} e^{-it \sqrt{z^2+m^2}} \frac{ \chi(z)}{z \sqrt{z^2+m^2}}[\mathcal{R}_0V \mathcal{R}_0v^*S_2D_3S_2 v \mathcal{R}_0V \mathcal{R}_0](z)(x,y)\,  dz.
\end{align}
Using the identity $\mathcal{R}_0 =  \mathcal{G}_0 +[\mathcal{R}_0- \mathcal{G}_0]$, and noting that  $S_2v\mG_1=0$ (see  Corollary~\ref{cor:S2 orth}) we can rewrite 
\begin{multline} \label{eq:genericmain}
\mathcal{R}_0V \mathcal{R}_0v^*S_2D_3S_2 v \mathcal{R}_0V \mathcal{R}_0 = \mathcal{R}_0V \mathcal{G}_0v^*S_2D_3S_2 v \mathcal{G}_0V \mathcal{R}_0 \\
      +  \mathcal{R}_0V \mathcal{G}_0v^*S_2D_3S_2 v [\mathcal{R}_0-\mathcal{G}_0-iz\mG_1]V \mathcal{R}_0    + \mathcal{R}_0V [\mathcal{R}_0-\mathcal{G}_0-iz\mG_1]v^*S_2D_3S_2 v\mathcal{G}_0V \mathcal{R}_0 \\
       +  \mathcal{R}_0V [\mathcal{R}_0-\mathcal{G}_0-iz\mG_1]v^*S_2D_3S_2 v [\mathcal{R}_0-\mathcal{G}_0-iz\mG_1]V \mathcal{R}_0. 
\end{multline}
Recall the expansions of \eqref{eq:R0exp_1} for $\mathcal R_0$ given in Lemma~\ref{lem:R0exp}, picking $\ell=0+$ we have
$$
 E_1(z)(x,y):=\frac{1}{z^2} \big[\mathcal{R}_0- \mathcal{G}_0-iz \mathcal{G}_1\big]   =  \mathcal G_2- iz\mathcal{G}_3+
	 \widetilde O_2\big (z^{1+} ( r^{2+}+r^{-1})\big). 
$$
Therefore, $E_1$ satisfies the hypothesis of Lemma~\ref{lem:genericbound} with $\alpha=2+$. Hence using Lemma~\ref{lem:genericbound} with $E_1$  as above, $E=S_2D_3S_2$, and $E_2=\mG_0$ we see that the   contribution of the second summand in \eqref{eq:genericmain} to \eqref{eq:genericmainint} is $O(\la t\ra^{-3/2})$ provided that $\beta>7$. The contribution of third and fourth summands can be handled in the same manner.  
 
Now we turn to the first term in the equation \eqref{eq:genericmain}.   By Lemma~\ref{lem:eproj} below, we have the identity
$\mathcal{G}_0v^*S_2D_3S_2 v \mathcal{G}_0=  \mathcal{G}_0V\mathcal{G}_0v^*S_2D_3S_2 v \mathcal{G}_0V\mathcal{G}_0= -2m  P_m$. Also note that, by \eqref{eq:G0S2}, we have
$$
\mG_0VP_m=P_mV\mG_0=-P_m.$$ 
Hence, the first term can be rewritten as
\begin{multline}\label{eq:RVGPm}
 \frac1{2m}\mathcal{R}_0V \mathcal{G}_0v^*S_2D_3S_2 v \mathcal{G}_0V \mathcal{R}_0 \\ =  - P_m +[ \mathcal{R}_0-\mathcal{G}_0]VP_m 
  + P_mV  [ \mathcal{R}_0-\mathcal{G}_0]-  [ \mathcal{R}_0-\mathcal{G}_0]VP_mV [ \mathcal{R}_0-\mathcal{G}_0].
\end{multline}
The contribution of $P_m$ in \eqref{eq:resint} is zero since the integral is an odd principal value integral. 
Note that the contributions of the last three terms to the Stone's formula is bounded for small $t$ by an argument similar to the proof of Lemma~\ref{lem:genericbound}.  
The following lemma takes care of the   contribution of the second and third terms to \eqref{eq:genericmain} when $t>1$.
\begin{lemma} Under the assumptions of Theorem~\ref{thm:main2}, 
 \be  \label{eq:PmG}
 K_1(t)(x,y):=   \int_{-\infty}^{\infty} e^{-it \sqrt{z^2+m^2} } \frac{ \chi(z)}{z\sqrt{z^2+m^2}}
  \Big[ P_mV [\mathcal{R}_0-\mathcal{G}_0]\Big](z)(x,y) dz   
\ee
  is a finite rank operator and  $ \|K_1(t)  \|_{L^{1} \rightarrow L^{\infty}} \les t^{-\f12}$. 
\end{lemma}
\begin{proof}
First of all note that $K_t$ is finite rank since $P_m$ is independent of $z$ and finite rank by Lemma~\ref{lem:eproj}. 
Therefore, it suffices to bound  \eqref{eq:PmG} by $t^{-\f12}$ uniformly in $x,y$.

Using  the definition of $\mathcal R_0(\lambda)$, and the equations \eqref{resolventex}, \eqref{eq:G0 def}, 
we have 
\begin{multline} \label{eq:r0-g0exp} 
[\mathcal{R}_0-\mathcal{G}_0](z) (y_1,y)  = [-i\alpha \cdot \nabla + 2mI_{uc}  ]\Big[ \frac{e^{iz|y-y_1|}-1}{4\pi |y-y_1|}\Big] + \frac{e^{iz|y-y_1|}}{4\pi |y-y_1|}[\sqrt{z^2+m^2}-m] \\
 = \Big[ \frac{ i\alpha \cdot (y-y_1)}{|y-y_1|^2} + 2mI_{uc} \Big] \Big[ \frac{e^{iz|y-y_1|}-1}{4\pi |y-y_1|}\Big]+ \Big[ z \frac{\alpha \cdot (y-y_1)}{|y-y_1|}+  \widetilde O_2(z^2) \Big] \frac{e^{iz|y-y_1|}}{4\pi |y-y_1|}.
\end{multline}
It is easy to see that the contribution of the last term is $O(\la t\ra^{-\f12})$ by a single application of Lemma~\ref{lem:vdc}. Note that the uniform bound in $x$ uses the boundedness of eigenfunctions, see Lemma~\ref{lem:esa2}. 
% Hence we need only bound the remaining terms uniformly in $y$.
 
Now we  estimate the contribution of the  first term in \eqref{eq:r0-g0exp}. Let 
\begin{multline}\widetilde{K}_t (y,y_1) := \int_{-\infty}^{\infty} e^{-it \sqrt{z^2+m^2}} \frac{ \chi(z)}{z \sqrt{z^2+m^2}}\Big[ \frac{ i\alpha \cdot (y-y_1)}{|y-y_1|^2} + 2mI_{uc} \Big] \Big[ \frac{e^{iz|y-y_1|}-1}{4\pi |y-y_1|}\Big]  dz \\
=i\int_{-\infty}^{\infty} e^{-it \sqrt{z^2+m^2}} \frac{ \chi(z)}{ \sqrt{z^2+m^2}}\Big[ \frac{ i\alpha \cdot (y-y_1)}{|y-y_1|^2} + 2mI_{uc} \Big]\int_0^{|y-y_1|} \frac{e^{izb}}{4\pi |y-y_1|} db dz.
\end{multline}
By Lemma~\ref{lem:vdc} we have 
$$
\Big| \int_{-\infty}^{\infty} e^{-it \sqrt{z^2+m^2}+izb} \frac{ \chi(z)}{ \sqrt{z^2+m^2}}dz \Big| \les t^{-\f12},
$$
uniformly in $b$. Therefore, by Fubini's theorem
$$
|\widetilde{K}_t (y,y_1)|\les t^{-\f12} \int_0^{|y-y_1|}(|y-y_1|^{-2}+|y-y_1|^{-1}) db \les t^{-\f12} (1+|y-y_1|^{-1}).
$$
Using the boundedness of eigenfunctions and the decay of $V$, we obtain 
$$|[P_mV \widetilde{K}_t] (x,y)|\les t^{-\f12},
$$ 
uniformly in $x, y$, which finishes the proof.
 \end{proof} 
% Note that with a very similar argument one can show that
%  \begin{multline*}
% \sup_{x,y \in \R^3} \Bigg| \int_{\infty}^{\infty} e^{-it \sqrt{z^2+m^2} } \frac{ \chi(z)}{z\sqrt{z^2+m^2}}
%  \Big[[\mathcal{R}_0-\mathcal{G}_0]VP_mV\mathcal{G}_0 \Big](x,y) dz - t^{-\f12}K_2(t)(x,y)  \Bigg| \les |t|^{-\f32}.
%\end{multline*}
%where $K_2(t)(x,y)=[\widetilde{K}_tVP_mV[(x,y)$ is a finite rank operator and  $\sup_t \|K_{2}(t)(x,y) \|_{L^{1} \rightarrow L^{\infty}} \les %1$. 

Now, we consider  the   contribution of the last term in \eqref{eq:RVGPm} to \eqref{eq:genericmain} when $t>1$.
\begin{lemma}  \label{lem:R0-G0}Under the assumptions of Theorem~\ref{thm:main2}  we have 
 \begin{align*} 
 \sup_{x,y \in \R^3} \Bigg| \int_{\infty}^{-\infty} e^{-it \sqrt{z^2+m^2} } \frac{\chi(z)}{z \sqrt{z^2+m^2}}
  \Big[[\mathcal{R}_0-\mathcal{G}_0]VP_mV [\mathcal{R}_0-\mathcal{G}_0]\Big](z)(x,y) dz \Bigg| \les t^{-\f32}.
\end{align*}
\end{lemma} 
\begin{proof}
Using \eqref{eq:r0-g0exp}, we have   
\begin{align*}
 \frac{ \mathcal{R}_0 - \mathcal{G}_0   }{z} &= M+N,\,\,\,\text{ where}\\
M(z)(x,x_1)& :=[-i\alpha \cdot \nabla] \bigg[  \frac{e^{iz|x-x_1|} -1  }{4\pi z|x-x_1|} \bigg] + [ \sqrt{z^2+m^2} -m] \frac{e^{iz|x-x_1|} }{4\pi z|x-x_1|},
\\ 
N(z,|x-x_1|)&:= 2mI_{uc} \bigg[  \frac{e^{iz|x-x_1|} -1   }{4\pi z|x-x_1| } \bigg].
\end{align*}
We will see that the operator $M$ satisfies suitable bounds. However, 
the operator $N$ does not, instead we need to use  that $ \mG_1VP_m = \frac{im}{2\pi}M_{uc} VP_m=0$ and $P_mV\mG_1=\frac{im}{2\pi} P_mVM_{uc}=0$. Hence,
%\begin{align} \label{eq:cancel}
%\int_{\R^3} [I_{uc}VP_mV] (x_1,x) f(x) dx = \int_{\R^3} [VP_mVI_{uc}] (y_1,y) f(y) dy = 0. 
%\end{align}
%Using \eqref{eq:cancel} 
we can replace the term $N(z,|x-x_1|)$ with 
$$\mathcal{N}(z)(x,x_1) = N(z,|x-x_1|) - N(z, \la x \ra)$$ 
on both sides of $VP_mV$. 

The following lemma contains the required bounds:
\begin{lemma}\label{lem:Mbounds} We have 
\begin{align}
 &\label{eq:boundM}  M(z)(x,x_1)  +  \mathcal{N}(z)(x,x_1) = \big[ \la x_1 \ra+|x-x_1|^{-1}\big] O\big( z\big), \\ 
& \label{eq:boundMprime} \partial_z \big(M(z)+\mathcal N(z)\big)(x,x_1)  = e^{iz|x-x_1|} \big[\la x_1\ra+|x-x_1|^{-1} \big] \widetilde O_1(1) + \la x_1\ra \widetilde O_1(1).  
\end{align}  
\end{lemma}
The proof of this lemma is given below. We finish the proof of Lemma~\ref{lem:R0-G0} using Lemma~\ref{lem:Mbounds}.

We start with applying integration by parts to the integral 
\begin{align} \label{eq:MM}
\int_{-\infty}^{\infty} e^{-it \sqrt{z^2+m^2} } \frac{ z \chi(z)}{\sqrt{z^2+m^2}} \big(M(z)+\mathcal N(z)\big)(x,x_1) \big(M(z)+\mathcal N(z)\big)(y,y_1) dz.
\end{align}
We only consider the case when the derivative falls on $\big(M(z)+\mathcal N(z)\big)(y,y_1)$. The other cases are similar. 
We therefore consider 
$$
t^{-1} \int_{-\infty}^{\infty} e^{-it \sqrt{z^2+m^2}  }   \chi(z)  \big(M(z)+\mathcal N(z)\big)(x,x_1) \partial_z  \big(M(z)+\mathcal N(z)\big)(y,y_1) dz.
$$ 
Using Lemma~\ref{lem:Mbounds},  we can write this integral as 
\begin{align*}
&t^{-1} (\la x_1\ra+|x-x_1|^{-1}) (\la y_1\ra+|y-y_1|^{-1})  \int_{-\infty}^{\infty} e^{-it \sqrt{z^2+m^2}+iz|y-y_1| }   \chi(z)  \widetilde O_1(z) dz \\
&+ t^{-1} (\la x_1\ra+|x-x_1|^{-1})  \la y_1\ra   \int_{-\infty}^{\infty} e^{-it \sqrt{z^2+m^2} }   \chi(z)  \widetilde O_1(z)  dz.
\end{align*}
Applying Lemma~\ref{lem:vdc} with the phase  $\phi(z)= -t \sqrt{z^2+m^2}+z|y-y_1|$ and the phase $\phi(z)= -t \sqrt{z^2+m^2} $ in each integral respectively, yields the bound
$$
	t^{-\f32} (\la x_1\ra+|x-x_1|^{-1}) (\la y_1\ra+|y-y_1|^{-1}). 
$$
This establishes the claim since $(\la x_1\ra+|x-x_1|^{-1})   |V (x_1)|\in L^2_{x_1}$ uniformly in $x$.
\end{proof}

\begin{proof}[Proof of Lemma~\ref{lem:Mbounds}]
We start with $M$. It is clear that second summand in the definition of $M$ satisfies \eqref{eq:boundM}  and \eqref{eq:boundMprime}.
Let $p=|x-x_1|$. The first summand in the definition of $M$ is (omitting the factors of $\frac{1}{4\pi}$) 
\begin{align} \label{eq:M}
 [-i\alpha \cdot \nabla] \Big[  \frac{e^{iz|x-x_1|} -1 }{z|x-x_1|} \Big]  = \frac{ i\alpha \cdot (x-x_1)}{p^2} \bigg[ \frac{e^{izp}-1}{zp} \bigg]  + \frac{ \alpha \cdot (x-x_1) e^{izp}}{p^2}=O(p^{-1}). 
\end{align}
Note that if $|z|p \gtrsim1 $,  this term is bounded by $|z|$.  
If $|z|p \les 1$, then we have $e^{izp} = 1+izp   -\frac12 z^2p^2 + \widetilde O_2(  z^3p^3)$. Plugging this into \eqref{eq:M}, we obtain
\begin{align} \label{eq:Msmall}
  [-i\alpha \cdot \nabla] \Big[  \frac{e^{iz|x-x_1|} -1 }{z|x-x_1|} \Big] =
[-i\alpha \cdot \nabla] \big[ i- \frac{zp}{2} +\widetilde{O}_2(z^2p^2)\big] =O(z). 
\end{align}
Taking the $z$ derivative of the first summand, we  obtain 
\begin{align*}
 \partial_z  [-i\alpha \cdot \nabla] \Big[  \frac{e^{iz|x-x_1|} -1 }{z|x-x_1|} \Big]&= \frac{i\alpha\cdot(x-x_1)}{p^2} \bigg[ \frac{e^{izp}i p}{zp} - \frac{e^{izp}-1}{z^2p} \bigg]  + \frac{ i \alpha \cdot (x-x_1) e^{izp}}{p} \\
 &= e^{izp} \frac{ i \alpha \cdot (x-x_1) }{p} \Big[   \frac{e^{-izp}-1+izp}{z^2p^2}  +1\Big] = e^{izp} \widetilde O_1 (1). 
\end{align*}
The last bound follows by noting that the numerator is $\widetilde O_1(z^2p^2)$ by a Taylor expansion when $zp\les 1$.

To obtain the bounds for $\mathcal N$, consider the function 
$$
	f(r)=\frac{e^{izr}-1 }{zr}.
$$
Then, ignoring the constant factors,  $\mathcal N(z)= f(p)-f(q)$, where $q:=\la x\ra$.
Using the mean value theorem  we have 
\begin{align}\label{eq:OboundN}
|\mathcal{N}(z)(x,x_1)|& \les \frac{|p-q|}{|z|} \max_r\bigg| \frac{ize^{izr} - e^{izr}+1 }{ r^2}   \bigg|  \les |z| \la x_1 \ra.
\end{align}
In the last inequality we used the fact  that  $\big( izre^{izr} - e^{izr}+1\big)=\widetilde O_1(z^2 r^2) $.

Further, 
\begin{align*}
\partial_z \mathcal{N}(z)(x,x_1) &= \frac{izpe^{izp}-e^{izp}+1}{z^2p}-\frac{izqe^{izq}-e^{izq}+1}{z^2q} \\
&= i e^{izp}  \f{1-e^{iz(q-p)}   }{z}  -\f {\mathcal N(z)}{z}  .
\end{align*}
By the mean value theorem we have
$$
\f{1-e^{iz(q-p)}   }{z} = \la x_1\ra \widetilde O_1(1). 
$$
Finally, note that the inequality  \eqref{eq:OboundN} and the calculation 
$$
\Big|\partial_z\f {\mathcal N(z)}{z}\Big| =  \Big|-{\f 1 {z^2}} \mathcal{N}(z)+ \f {\partial_z \mathcal N(z)}{z} \Big| \les \f{\la x_1\ra }{|z|}
$$
imply that  
$$
 \f {\mathcal N(z)}{z}=\la x_1\ra \widetilde O_1(1).
$$
\end{proof}

This completes the proof of Proposition~\ref{prop:res2}.\end{proof}
 
\begin{proof}[Proof of Theorem~\ref{thm:main2}] 
The proof follows as in the proof of Theorem~\ref{thm:main1} using Proposition~\ref{prop:res2} instead of Proposition~\ref{prop:res1}. 
\end{proof}

\begin{rmk}
	
	This method  also applies to the analysis of the  Schr\"odinger 
	operator considered in \cite{ES} and \cite{Yaj3}.  In particular, it implies  that the $t^{-\f12}$ term   is a time dependent finite rank 
	operator when zero is not a regular point of the spectrum. This gives an alternative proof to Yajima's theorem in \cite{Yaj3}. In \cite{ES}, such a result was obtained only in the case when there is a resonance of the first kind.   
	
\end{rmk}

\section{Classification of threshold spectral subspaces} \label{sec:esa}

 \begin{lemma}\label{lem:esa1}
Assume $|v(x)| \les \la x \ra ^{-{\f32}-} $. Then $\phi \in S_1 L^2(\R^3) \setminus \{0\} $ if and only if $\phi= Uv \psi $ for some  $\psi \in L^{2,-{\f 12}-}(\R^3)\setminus \{0\} $ which  is a distributional solution of  $ ( D_m+V-mI)\psi=0 $. Furthermore, $\psi=- \mathcal{G}_0v^{*}\phi$ and $\psi$ is a bounded function.
\end{lemma} 
  \begin{proof} If  $\phi \in S_1 L^2(\R^3)  \setminus \{0\}, $ then by Definition~\ref{def:res defn}, $(U+ v \mathcal{G}_0 v^{*}) \phi =0$.   Since $U^2=I$,  
  \begin{align}\label{S1G0}
 \phi= -U v \mathcal{G}_0 v^{*} \phi = U v \psi,\,\, \text{ where }\,\,  \psi:=-  \mathcal{G}_0 v^{*} \phi.
   \end{align}
 Using \eqref{eq:G0 def}  and \eqref{eq:Dmpm} with $\lambda=m$, we obtain
  \begin{multline}\label{eq:eigeniden}
(D_m - mI) \psi   = -(D_m - mI)\mathcal{G}_0 v^{*}\phi = -(D_m - mI)(D_m +mI)G_0v^{*}\phi \\
 = \Delta G_0 v^{*}\phi = - v^{*} \phi =-v^*Uv\psi =-V\psi.
\end{multline}  
Therefore,    $ ( D_m +V -mI)\psi=0 $. In the fourth equality above, 
we used  the fact that $\Delta G_0 v^{*}\phi=-v^{*}\phi$ holds since  $v^{*}\phi \in L^{2,\f32+}$, see Lemma~2.4 in \cite{JenKat}.  

Now we prove that   $\psi \in L^{2,-{\f 12}-}(\R^3) $. Note that 
\begin{multline} 
 \psi =- [ - i\alpha \cdot \nabla + 2m I_{uc} ] G_0 v^{*} \phi 
  =- [ - i\alpha \cdot \nabla + 2m I_{uc} ]  \frac{1}{4 \pi} \int_{\R^3} \frac{v^{*}(y) \phi (y)}{|x-y|} \\
  =   \frac{1}{4 \pi} \int_{\R^3}  i \alpha \cdot(x-y) \frac{v^{*}(y) \phi (y)}{|x-y|^3} - 2m I_{uc}  \frac{1}{4 \pi} \int_{\R^3} \frac{v^{*}(y) \phi (y)}{|x-y|} 
 =:  \psi_1+ \psi_2. \label{eq:g1g2}
 \end{multline}
 Since the integrals in equation can be bounded by   fractional integral operators, we can use Lemma 2.3 in \cite{Jen}. We have $\psi_1\in L^2(\R^3) \subseteq L^{2,-{\f 12}-}(\R^3)  $ provided $|v(x)| \les \la x \ra ^{-1} $; and  $\psi_2\in L^{2,-{\f 12}-}(\R^3)$ provided $|v(x)| \les \la x \ra ^{-{\f32}-} $. 
 
 Conversely, assume that $\phi= Uv \psi $ for some $\psi \in L^{2,-{\f 12}-}(\R^3) $ satisfying $ (D_m+V-mI)\psi=0 $. 
Then $\phi\in L^{2,1+}$, and by a calculation similar to \eqref{eq:eigeniden}, we have 
$$
(D_m - mI) \psi   = -V\psi= - v^{*} \phi= \Delta G_0 v^{*}\phi = -(D_m - mI)(D_m +mI)G_0v^{*}\phi =  -(D_m - mI)\mathcal{G}_0 v^{*}\phi.
$$ 
Thus, also using \eqref{eq:Dmpm} with $\lambda=-m$, we have 
$$ \Delta (\psi +  \mathcal{G}_0 v^{*}\phi)=(D_m+mI)(D_m-mI)(\psi +  \mathcal{G}_0 v^{*}\phi)  = 0. 
$$
Noting that $\psi +  \mathcal{G}_0 v^{*}\phi\in L^{2,-{\f 12}-}(\R^3)$, we conclude that (see \cite{JenKat})
 $  \psi +   \mathcal{G}_0 v^{*}\phi=0.$
Notice that this also implies that the free Dirac has no threshold resonances.
Therefore,
$$
 (U+ v \mathcal{G}_0 v^{*}) \phi = v\psi  +v\mathcal{G}_0 v^* \phi=0,
$$
and hence $\phi\in S_1L^2$.
  
Since $\phi= Uv \psi $, if $\phi\neq 0 $, then $ \psi\neq 0$. The reverse implication  follows from $\psi=- \mathcal{G}_0 v^{*} \phi$.

Finally, using \eqref{S1G0} we have $\psi =-\mG_0V\psi $. Iterating this identity we obtain $\psi =\mG_0V\mG_0V\psi $. Therefore, by a calculation identical to the one in Remark~\ref{rmk:potential}, we see that $\psi $ bounded.
\end{proof}

 \begin{lemma}\label{lem:esa2} Suppose $|v(x)| \les \la x \ra ^{-{\f32}-} $. Fix $\phi=Uv\psi\in S_1L^2$, where $\psi \in L^{2,-{\f 12}-}(\R^3)\setminus \{0\} $  is a distributional solution of  $ ( D_m+V-mI)\psi=0 $. Then $\phi \in S_2  L^2(\R^3) $ if and only if   $\psi \in  L^2(\R^3)$.  Moreover, any threshold eigenfunction, $\psi$, is a bounded function.
  \end{lemma}
  
  \begin{proof} The boundedness of $\psi$ and the equality $ ( D_m+V-mI)\psi=0 $ were obtained in the   previous lemma.  
  First note that if $\phi  \in S_2 L^2(\R^3) $, namely $ S_1v\mathcal{G}_1v^*\phi  =0$, then 
  \begin{align*}
    0= \la vM_{uc} v^*\phi ,\phi  \ra =  \la M_{uc} v^*\phi , M_{uc}v^*\phi  \ra_{\C^4} = \|M_{uc} v^* \phi \|^2_{\C^4}, 
  \end{align*} 
Hence, $M_{uc} v^*\phi =0$. It is also clear that  if $M_{uc} v^*\phi =0$, then $\phi\in S_2L^2$.
  
 Also note that in the proof of Lemma~\ref{lem:esa1}, we showed that $\psi=\psi_1+\psi_2$ and  $\psi_1 \in L^2(\R^3) $. 
   Therefore it suffices to prove that $M_{uc}v^*\phi=0$ if and only if $\psi_2\in L^2$. 
Recalling \eqref{eq:g1g2} we can write  $\psi_2$  as
\be\label{eq:psi2}
  \psi_2(x) =    \frac{m}{2 \pi} I_{uc}  \int_{\R^3} v^{*}(y) \phi (y) \bigg[ \frac{1}{|x-y|} - \frac{1}{\la x\ra} \bigg] dy
  + \frac{m}{2 \pi \la x \ra}  [M_{uc}v^* \phi]. 
\ee
Using  \cite[Lemma~6]{ES} we see that the first integral above is in $L^2(\R^3)$. Since $\f1{\la x\ra}\not \in L^2(\R^3)$, we conclude that $\psi_2\in L^2$ if and only if $ M_{uc}v^* \phi=0$.
  \end{proof}
A useful consequence of this proof is the following orthogonality condition and the fact that the rank of $S_1-S_2$ is at most two.
\begin{corollary}\label{cor:S2 orth}
	We have the identities
	\begin{align*}
		S_2v\mathcal G_1=\mathcal G_1v^*S_2=\frac{m}{2\pi}M_{uc}v^* S_2=\frac{m}{2\pi}S_2vM_{uc}=0.
	\end{align*}
\end{corollary}  

\begin{corollary}\label{cor:2d}
Assume $|v(x)| \les \la x \ra ^{-{\f32}-}$.  Then the rank of $S_1-S_2$ is at most two.
\end{corollary}  
\begin{proof}
	We consider the representation in \eqref{eq:g1g2}.  We have already shown that $\psi_1\in L^2$. By \eqref{eq:psi2} and the discussion following it, we can write  $\psi_2$ as 
$$
 \frac{m}{2 \pi \la x \ra}  [M_{uc}v^* \phi]+O_{L^2}(1)=\frac{1}{ \la x\ra} (a_1, a_2, 0, 0)^T+O_{L^2}(1).
$$	
 The constants $a_j= \frac{m}{2\pi} \int_{\R^3} [v^*(y)\phi(y)]_j \, dy$ are finite by the assumed decay of $v^*$.
\end{proof}

  \begin{lemma}\label{lem:esa3}
Assume $|v(x)| \les \la x \ra ^{-{\f52}-}$. Then $S_2v\mathcal{G}_2v^*S_2$ is invertible  as an operator in $S_2 L^2(\R^3) $.  
   \end{lemma}
\begin{proof}
Since $S_2v\mathcal{G}_2v^*S_2$ is a compact operator it is enough to show that its kernel is empty. 
Assume that for some $\phi \in S_2 L^2(\R^3)$, $S_2v\mathcal{G}_2v^*S_2\phi=0$, i.e., $ \la \mathcal{G}_2 v^*\phi , v^*\phi  \ra =0$. By Corollary~\ref{cor:S2 orth} $ \mathcal{G}_1 v^*\phi =0$. Using these equalities in \eqref{eq:R0exp_1}, under the decay condition on $|v(x)|$ 
\begin{align} \label{eq:verb}
0= \la \mathcal{G}_2 v^*\phi , v^*\phi  \ra = - \lim_{z \rightarrow 0} { \f 1 {z^2}}\la [ \mathcal{R}_0(\lambda) - \mathcal{G}_0] v^* \phi , v^*\phi  \ra 
\end{align} 
where $\lambda= \sqrt{z^2+m^2}$. The following equality holds for $z= i\omega$ and $ 0 < \omega \ll m$, 

$$\f 1 {z^2} \la [ \mathcal{R}_0(\lambda) - \mathcal{G}_0] v^* \phi , v^*\phi  \ra = \int_{\R^3} \big\la K(\omega,\xi) \widehat{v^*\phi (\xi)}, \widehat{v^*\phi (\xi)} \big\ra _{\C^4} d\xi \ , $$
where
\begin{multline*}
 K(\omega,\xi) =
   \frac{1}{\omega^2 |\xi|^2} \left(\begin{array}{cccc}
2m & 0 & \xi_3 & \bar{\eta}       \\
0 & 2m & \eta &  -\xi_3      \\
\xi_3 & \bar{\eta}  & 0 & 0      \\
\eta &  -\xi_3 & 0 & 0
\end{array}\right)\\
 -\frac{1}{\omega^2(\omega^2+|\xi|^2)}
              	\left(\begin{array}{cccc}
m+\sqrt{m^2-\omega^2} & 0 & \xi_3 & \bar{\eta}      \\
0 & m+\sqrt{m^2-\omega^2}  & \eta & -\xi_3       \\
 \xi_3 &  \bar{\eta} & \sqrt{m^2-\omega^2}-m &0       \\
 \eta & -\xi_3 & 0 & \sqrt{m^2-\omega^2}-m
\end{array}\right).
 \end{multline*}
Here $\xi=( \xi_1, \xi_2, \xi_3)$ and $\eta = \xi_2 + i\xi_1$. 

Let $\tau:=\frac{|\xi|^2}{\omega^2}(m-\sqrt{m^2-\omega^2})$, then $ K(\omega,\xi)$ can be written as 
\begin{align}
 \frac{1}{|\xi|^2(\omega^2+|\xi|^2)}
  \left(\begin{array}{cccc}
2m+\tau &0&\xi_3&\eta   \\
0 & 2m+\tau & \eta &  -\xi_3        \\
\xi_3 & \bar{\eta} & \tau & 0        \\
\eta &  -\xi_3  & 0 &\tau
\end{array}\right).
\end{align} 
The eigenvalues of $K(\omega,\xi)$ are
 \begin{align*}
 &\lambda_{1,2}=\frac{m+\tau+\sqrt{m^2+|\xi|^2}}{|\xi|^2(\omega^2+|\xi|^2)},  \qquad \lambda_{3,4}= \frac{m+\tau-\sqrt{m^2+|\xi|^2}}{|\xi|^2(\omega^2+|\xi|^2)}. 
 \end{align*} 
 
Note that the eigenvalues are real and  for any $\xi\neq0$ they are positive. Hence, $ K(\omega,\xi)$ self-adjoint and positive definite for any $\xi\neq0$. One can also check that the eigenvalues are nonincreasing functions of $\omega \in (0,m)$. Hence, we can use monotone convergence theorem and take the limit into the integral \eqref{eq:verb} to obtain
$$ 0= \lim_{\omega\rightarrow 0^+}\int_{\R^3} \la K(\omega,\xi)\hat{v^*\phi (\xi)} , \hat{v^*\phi (\xi)} \ra_{\C^4} d\xi \ =\int_{\R^3} \la K(0,\xi)\hat{v^*\phi (\xi)} , \hat{v^*\phi (\xi)} \ra_{\C^4} d\xi$$
where
\begin{align*}
K(0,\xi)= \frac{1}{|\xi|^4} \left(\begin{array}{cccc}
2m+\frac{|\xi|^2}{2m} & 0 & \xi_3 & \bar{\eta}      \\
0 & 2m+\frac{|\xi|^2}{2m} &  \eta& -\xi_3   \\
\xi_3 & \bar{\eta} & \frac{|\xi|^2}{2m} & 0       \\
 \eta &  -\xi_3 & 0 & \frac{|\xi|^2}{2m}
\end{array}\right).
   \end{align*}
Note that this matrix is also self-adjoint and positive definite. Therefore, $\hat{v^*\phi (\xi)}=0$. Since $v^*\phi (\xi)$ has $L^1$ entries, $v^*\phi  =0$. Recall that the fact that $\phi  \in S_1 L^2(\R^3) $ implies that $\phi = Uv^*\psi$  for $\psi= - \mathcal{G}_0v^*\phi $. Hence, we conclude that $\phi =0$.
\end{proof}  
In addition, one can see that 
\begin{align*}
K(0,\xi)= \frac{1}{2m}  \frac{1}{|\xi|^4} \left(\begin{array}{cccc}
2m & 0 & \xi_3 & \bar{\eta}      \\
0 & 2m &  \eta& -\xi_3   \\
\xi_3 & \bar{\eta} & 0 & 0       \\
 \eta &  -\xi_3 & 0 & 0
\end{array}\right)^2.
   \end{align*}
Therefore, for any $\phi \in S_2L^2$ we have 
\begin{align}\label{G0 to G2 ident} 
\la \mathcal{G}_2 v^*\phi  , v^*\phi  \ra =-\frac{1}{2m} \la  \mathcal{G}_0 v^*\phi , \mathcal{G}_0 v^*\phi  \ra.
\end{align}
%\end{proof}
\begin{lemma}\label{lem:eproj}
The operator $P_m = -\frac{1}{2m}\mathcal{G}_0V\mathcal{G}_0v^*S_2D_3S_2v\mathcal{G}_0V\mathcal{G}_0$ equals the finite rank, orthogonal projection in $ L^2(\R^3) $ onto the eigenspace of $H = D_m + V$ at threshold $m$.  
\end{lemma}   
\begin{proof} 
First recall that $S_2 \leq S_1 $ is finite dimensional. 
 Using \eqref{S1G0}  we have $S_2= -S_2v\mathcal{G}_0v^*U$ and consequently
\be\label{eq:G0S2}
S_2v\mathcal{G}_0V\mathcal{G}_0 = S_2v\mathcal{G}_0v^*U v \mathcal{G}_0 = - S_2v\mathcal{G}_0.
\ee
Similarly, $ \mathcal{G}_0V\mathcal{G}_0v^*S_2 = -\mathcal{G}_0v^*S_2$. Therefore, $P_m = -\frac{1}{2m}\mathcal{G}_0v^*S_2D_3S_2v\mathcal{G}_0$.
 
Let $\{\phi_j\}_{j=1}^N$ be an orthonormal basis for the $S_2  L^2(\R^3)$, the range of $S_2$.  Then, by Lemma~\ref{lem:esa1}, we have
\begin{align}\label{psi j eqn}
	\phi_j=Uv\psi_j,\,\,\,\,\psi_j=-\mathcal G_0v^*\phi_j, \qquad 1\leq j\leq N,
\end{align}
where $\psi_j\in L^2$, $j=1,2,\ldots, N$, are eigenvectors. This implies that he range of $P_m$ is contained in the span of $\{\psi_j \}_{j=1}^N$.

Since  $\{\phi_j\}_{j=1}^N$ is linearly independent, we have that $\{\psi_j\}_{j=1}^N$ is  linearly independent, and hence it is a basis for $m$ energy eigenspace.
Using the orthonormal basis for $S_2 L^2(\R^3) $, we have that for any $f\in L^2 $,
$S_2f=\sum_{j=1}^N \la f,\phi_j\ra \phi_j$. Therefore, we have
\begin{align}\label{S2 sum}
	S_2 v \mathcal G_0f=\sum_{j=1}^N \la f,\mathcal G_0v^*\phi_j\ra \phi_j=-\sum_{j=1}^N \la f,\psi_j\ra \phi_j.
\end{align}
We claim that, for each $i_0,j_0\in\{1,2,\ldots,N\},$
$$
\big\la \psi_{i_0}, P_m \psi_{j_0}\big\ra =\big\la \psi_{i_0}, \psi_{j_0}\big\ra.
$$
This implies the range of $P_m$ is equal to the span of $\{\psi_j \}_{j=1}^N$ and that $P_m$ is the identity operator in the range of $P_m$. Since $P_m$ is self-adjoint the assertion of the lemma holds.

Recall $D_3:=(S_2v\mathcal G_2v^*S_2)^{-1}$. Let $A=\{A_{ij}\}_{i,j=1}^N$, $B=\{B_{ij}\}_{i,j=1}^N$ be the matrix representations of $S_2v\mathcal G_2v^*S_2$ and $D_3$ with respect to the
orthonormal basis $\{\phi_j\}_{j=1}^N$ of $S_2$.
Using \eqref{G0 to G2 ident} and polarization,
\begin{align*}
	A_{ij}&=\la \phi_j,S_2v\mathcal G_2v^*S_2\phi_i\ra=-\frac{1}{2m} \la \mathcal G_0v^*\phi_j,\mathcal G_0v^*\phi_i\ra  = - \frac{1}{2m} \la \psi_j,\psi_i\ra,\\
	B_{ij}&= A^{-1}_{ij}=\la \phi_j,D_3 \phi_i\ra.
\end{align*}
Using this and \eqref{S2 sum},  we have 
\begin{multline*}
\big\la \psi_{i_0}, P_m \psi_{j_0}\big\ra= -\frac1{2m}\Big\la S_2v\mathcal G_0  \psi_{i_0},D_3 S_2v\mathcal G_0 \psi_{j_0} \Big\ra \\ =
-\frac1{2m}
\Big\la \sum_{i=1}^N \la \psi_{i_0},\psi_i\ra \phi_i, D_3 \sum_{j=1}^N \la \psi_{j_0},\psi_j\ra \phi_j \Big\ra 
=-\frac1{2m}\sum_{i,j=1}^N \la \psi_{i_0},\psi_i\ra \la  \psi_j, \psi_{j_0} \ra \big\la \phi_i,D_3  \phi_j \big\ra\\ = -2m \sum_{i,j=1}^N  A_{i,i_0} B_{j,i} A_{j_0,j} = -2m A_{j_0,i_0}=   \la \psi_{i_0},\psi_{j_0}\ra.
\end{multline*}
This finishes the proof of the claim and the lemma.
\end{proof}

\end{document}